\documentclass[12pt]{article}
\usepackage{graphicx, times}
\usepackage{amsfonts}
\usepackage{amsmath}
\usepackage{amsthm}
\usepackage{amssymb}
\usepackage{fancyhdr}
\usepackage{indentfirst}
\usepackage{titlesec}
\usepackage{hyperref}
\usepackage{abstract}
\usepackage{color}

\usepackage[normalem]{ulem}
\usepackage{soul}

\newcommand{\al}{\alpha}
\newcommand{\be}{\beta}
\newcommand{\V}{\mathcal{V}}
\newcommand{\R}{\mathbb{R}}
\newcommand{\N}{\mathbb{N}}
\newcommand{\mH}{\mathcal{H}}
\newcommand{\F}{\mathcal{F}}
\newcommand{\bF}{\mathbf{F}}

\newcommand{\lan}{\langle}
\newcommand{\ran}{\rangle}
\newcommand{\la}{\lambda}
\newcommand{\lc}{\scalebox{1.5}{$\llcorner$}}
\newcommand{\La}{\Lambda}

\newcommand{\Si}{\Sigma}
\newcommand{\de}{\delta}

\newcommand{\ep}{\epsilon}
\newcommand{\pr}{\prime}

\newcommand{\Ga}{\Gamma}
\newcommand{\ga}{\gamma}

\newcommand{\ti}{\tilde}

\newcommand{\X}{\mathfrak{X}}

\newcommand{\Diff}{\textrm{Diff}_0}
\newcommand{\Js}{\mathfrak{Js}}

\newcommand{\rom}[1]{\expandafter\romannumeral #1}
\newcommand{\Rom}[1]{\uppercase\expandafter{\romannumeral #1}}

\newcommand{\Gau}{\frac{1}{4\pi}e^{-\frac{|x|^2}{4}}}

\newcommand{\beq}{\begin{equation}}
\newcommand{\eeq}{\end{equation}}
\newcommand{\bite}{\begin{itemize}}
\newcommand{\eite}{\end{itemize}}
\newcommand{\benu}{\begin{enumerate}}
\newcommand{\eenu}{\end{enumerate}}

\begin{document}

\newtheorem{theorem}{Theorem}[section]
\newtheorem{proposition}[theorem]{Proposition}
\newtheorem{corollary}[theorem]{Corollary}

\newtheorem{claim}{Claim}

\theoremstyle{remark}
\newtheorem{remark}[theorem]{Remark}

\theoremstyle{conjecture}
\newtheorem{conjecture}[theorem]{Conjecture}

\theoremstyle{definition}
\newtheorem{definition}[theorem]{Definition}

\theoremstyle{plain}
\newtheorem{lemma}[theorem]{Lemma}

\numberwithin{equation}{section}

%
%
%
\titleformat{\part}
{\normalfont\large\bfseries}
{\partname\vspace{1em} \thepart:}
{1em}
{\large}
%
%
 \titleformat{\section}
   {\normalfont\bfseries\large\filcenter}
   {\arabic{section}}
   {12pt}{}
 \titleformat{\subsection}
   {\normalfont\bfseries}
   {\arabic{section}.\arabic{subsection}}
   {11pt}{}
%
%
%

\pagestyle{headings}
\renewcommand{\headrulewidth}{0.4pt}

\title{\textbf Entropy of closed surfaces and min-max theory}
\author{Daniel Ketover and Xin Zhou\footnote{The first author is partially supported by an NSF Postdoctoral Fellowship.  The second author is partially supported by NSF grant DMS-1406337.}}
\date{\today}
\maketitle

\pdfbookmark[0]{}{beg}

\renewcommand{\abstractname}{}    
\renewcommand{\absnamepos}{empty} 
\begin{abstract}
\textbf{Abstract:} Entropy is a natural geometric quantity measuring the complexity of a surface embedded in $\mathbb{R}^3$.  For dynamical reasons relating to mean curvature flow, Colding-Ilmanen-Minicozzi-White conjectured that the entropy of any closed surface is at least that of the self-shrinking two-sphere.  We prove this conjecture for all closed embedded $2$-spheres.  Assuming a conjectural Morse index bound (announced recently by Marques-Neves), we can improve the result to apply to all closed embedded surfaces that are not tori.  Our results can be thought of as the parabolic analog to the Willmore conjecture and our argument is analogous in many ways to that of Marques-Neves on the Willmore problem.  The main tool is the min-max theory applied to the Gaussian area functional in $\mathbb{R}^3$ which we also establish.  To any closed surface in $\mathbb{R}^3$ we associate a four parameter canonical family of surfaces and run a min-max procedure.  The key step is ruling out the min-max sequence approaching a self-shrinking plane, and we accomplish this with a degree argument.  To establish the min-max theory for \ $\mathbb{R}^3$ with Gaussian weight, the crucial ingredient is a tightening map that decreases the mass of non-stationary varifolds (with respect to the Gaussian metric of $\R^3$) in a continuous manner.

\end{abstract}




\section{Introduction}
Surprisingly, the resolution of the Willmore Conjecture by F.C. Marques and A. Neves \cite{MN12} hinges on asking and answering the following question:  
\\
\\
{\em In round $\mathbb{S}^3$, what is the non-equatorial embedded minimal surface with smallest area?} 
\\
\\
Using min-max theory, \cite{MN12} proved that the answer is the Clifford torus.  This swiftly led to a proof of the Willmore conjecture. In this paper, we address the analogous question for singularity models for the mean curvature flow and cast the question in terms of min-max theory.  The min-max argument will also allow us to partly prove a 'parabolic' version of the Willmore conjecture. Let us first introduce the relevant objects of study.

We will consider $\R^3$ endowed with the Gaussian metric $g_{ij}=e^{-|x|^2/4}\delta_{ij}$.   Minimal surfaces in this metric are called \emph{self-shrinkers} and arise as blowup limits at singularities of the mean curvature flow (MCF).   The simplest self-shrinkers are the family of flat planes passing through the origin (parameterized by $\mathbb{RP}^2$, their unoriented normal).  These surfaces arise as blowup limits at {\em smooth} points along the MCF.   Other examples include the round two-sphere (suitably normalized), which is the type of singularity one encounters as a sphere shrinks to a point by MCF, and the cylinder, which can be thought to model a neck-pinch type singularity.   

 For a smooth surface $\Sigma\subset\R^3$, we define its {\em Gaussian area}:
\begin{equation}\label{Gaussian area}
F(\Sigma)= \frac{1}{4\pi}\int_\Sigma e^{-|x|^2/4}dx.
\end{equation}
\indent  
Critical points for Gaussian area are precisely the self-shrinkers \cite[Proposition 3.6]{CM12}.  Following Colding-Minicozzi \cite{CM12}, for any surface $\Sigma\subset\R^3$, we define the \emph{entropy} of $\Sigma$ to be the supremum of Gaussian areas over all translations ($t\in\mathbb{R}^3$) and rescalings ($s\in[0,\infty)$) of $\Sigma$:
\begin{equation}\label{definition of entropy}
\lambda(\Sigma):=\sup_{t,s}F(s(\Sigma-t)),
\end{equation}
where the translated and dilated surface $s(\Sigma-t)$ is defined as 
\begin{equation}
s(\Sigma-t):=\{s(x-t)|x\in\Sigma\}.
\end{equation}
One key property of a self-shrinker is that its entropy is equal to its Gaussian area (c.f. \cite[\S 7.2]{CM12} and \S \ref{entropy achieved at (0, 1)}), analogous to the fact that the Willmore energy is equal to area for a minimal surface in $\mathbb{S}^3$. 

Just as for minimal surfaces in a smooth $3$-manifold, on a self-shrinker $\Sigma$ one can consider the second variation of Gaussian area. For any smooth function $\phi$ defined on $\Sigma$, if $\mathbf{n}$ is a choice of unit normal vector field along $\Sigma$
we consider \cite[Theorem 4.1]{CM12}
\begin{equation}
\frac{\partial^2}{\partial s^2}\Big|_{s=0} F(\Sigma+\mathbf{n}\phi s) =\frac{1}{4\pi} \int_\Sigma \phi(L\phi) e^{-|x|^2/4}dx,
\end{equation}
where $L$ is the second-order elliptic operator defined on $\Sigma$ as 
 \begin{equation}
L = \Delta+|A|^2-\frac{1}{2}\langle x,\nabla (\cdot)\rangle + \frac{1}{2}.
\end{equation}

\indent
The \emph{Morse index} of a self-shrinker is the dimension of a maximal subspace of variations on which $L$ is negative definite. A self-shrinker is stable if the Morse index is $0$. It is easily seen that there is no stable closed self-shrinker \cite[\S 0.2]{CM12}. Moreover, Colding-Minicozzi \cite{CM12} have classified those self-shrinkers with Morse index at most $4$: they are $\mathbb{R}^2$, $\mathbb{S}^2$ and $\mathbb{S}^1\times\mathbb{R}$.  The entropies of these surfaces were computed by Stone \cite{ST94}: $\lambda(\mathbb{R}^2)=1$, $\lambda(\mathbb{S}^2)\approx1.47$ and  $\lambda(\mathbb{S}^1\times\mathbb{R}^1)\approx 1.5203$.   Recently Brendle \cite{BR} has classified shrinkers with genus $0$: they are either the plane, cylinder or sphere.  \\
\indent
The self-shrinker of smallest area is the flat plane, and one can ask (as Marques-Neves \cite{MN12} ask in $\mathbb{S}^3$) which non-flat self-shrinker is simplest, i.e. has smallest area above the flat planes.  Using dynamical methods, Colding-Ilmanen-Minicozzi-White \cite{CIMW} proved that the {\it compact} self-shrinker with smallest area above the plane is the two sphere.  They further conjectured:
\begin{conjecture}\cite{CIMW}\label{cimw}
 For any smooth closed embedded surface $\Sigma$ in $\mathbb{R}^3$, 
\begin{equation}\label{min}
\lambda(\Sigma)\geq\lambda(\mathbb{S}^2).
\end{equation}   
\end{conjecture}
\noindent
The motivation as articulated in \cite{CIMW} for Conjecture \ref{cimw} is dynamical: starting with a closed surface, since the entropy is decreasing along the MCF and invariant under dilation and translation, by blowing up at a first singularity one should be able to prove \eqref{min}.  Necessary for carrying this out, however, is knowing that the {\em non-compact} shrinkers also have area at least that of the two-sphere, which did not follow from their arguments.  

 Our main result is a min-max proof of this conjecture for closed $2$-spheres:
\begin{theorem}\label{main}
Let $\Sigma\subset\mathbb{R}^3$ be a smooth closed embedded $2$-sphere.  Then $$\lambda(\Sigma)\geq\lambda(\mathbb{S}^2).$$
\end{theorem}
\noindent

In many ways Conjecture \ref{cimw} can be thought of as a 'parabolic' analog of the Willmore conjecture where the entropy plays the role of the Willmore energy.   We will use the min-max theory applied to the Gaussian area functional.  The key discovery is that any closed embedded surface in $\mathbb{R}^3$ that is not a torus gives rise to a natural and non-trivial element in the relative homotopy group $\pi_4(\mathcal{Z}_2(\mathbb{R}^3),\{\mbox{Affine Planes}\})$ whose corresponding min-max critical point is the round sphere.  In the proof of the Willmore Conjecture, the essential discovery was similarly the existence of a non-trivial element in $\pi_5(\mathcal{Z}_2(\mathbb{S}^3),\{\mbox{Geodesic Spheres}\})$, whose corresponding critical point is the Clifford torus.  At each stage of the argument we encounter issues very similar to those handled by Marques-Neves \cite{MN12}.

\begin{remark}
Assuming a conjectural index bound for min-max limits that has recently been announcd by Marques-Neves, Theorem \ref{main} can be extended to all closed surfaces with genus not equal to $1$.  
\end{remark}
\begin{remark}
Recently Bernstein-Wang \cite{BW},\cite{BW2} have proved Conjecture \ref{cimw} using a completely different method.  They have further shown that non-compact shrinkers in $\R^3$ have Gaussian area at least that of the round two-sphere.  The point of this paper is to give a completely independent proof using entirely different methods.  This is also the first case beyond the Willmore Conjecture where the min-max method and index/genus classification theorems can give lower bounds for area of minimal surfaces.
\end{remark}
\vspace{1.5em}

Let us briefly sketch the argument for Theorem \ref{main}.  Starting out with a smooth two-sphere $\Sigma$, one can consider the canonical four parameter family of genus 0 surfaces given by $\Sigma_{s,t}=s(\Sigma-t)$ of translates and dilates of $\Sigma$.  This is a natural sweepout to consider since translates and dilates are the $4$ unstable directions that are always present for any self-shrinker.  By definition, the entropy $\lambda(\Sigma)$ is greater than or equal to the area of any surface in this family.  The entropy thus gives an upper bound for the width of this canonical homotopy class of sweepouts.   The min-max theory that we develop for the Gaussian area (Theorem \ref{main min-max theorem}) gives then a self-shrinker $\Gamma$ of area at most $\lambda(\Sigma)$  and genus $0$ (by Simon-Smith \cite{Sm82} the genus cannot grow).  By the classification of Brendle \cite{BR} this shrinker must be the plane, sphere or cylinder.  The crucial fact is that just as in the proof of the Willmore Conjecture, the boundary of our sweepout records the genus of $\Sigma$ and we can then use a topological degree argument to rule out getting the plane.   Thus we can show $\Gamma = \mathbb{S}^2$ and we obtain $F(\mathbb{S}^2)=\lambda(\mathbb{S}^2)\leq\lambda(\Sigma)$.
\\
\indent
We avoid many of the technical difficulties present in the proof of the Willmore Conjecture since we are able to use sweepouts where all surfaces have the same genus and vary smoothly (though as we will see they degenerate to planes or points as one approaches the boundary of the parameter space).   Nonetheless, carrying out the min-max construction for the Gaussian area is a subtle problem because the manifold is non-complete and the curvature blows up at infinity. We refer to Section \ref{Min-max theory for Gaussian area} for more details.

Throughout this paper, by \emph{plane} we will mean a plane passing through the origin.  An \emph{affine plane} is then a plane not necessarily passing through the origin.  (As we will see, the distinction between planes and affine planes is completely analogous to the distinction between great spheres and geodesic spheres that appears throughout the proof of the Willmore conjecture)
\\

The paper is organized as follows. In Section \ref{Min-max theory for Gaussian area}, we introduce our min-max theory for Gaussian area. Then the paper is separated into two parts. In Part \ref{entropy bound} (including \S \ref{canonical family} and \S \ref{min-max argument and proof of the main theorem}), we use the min-max theory to prove the entropy bound theorem. In Part \ref{min-max for Gaussian area} (including \S \ref{preliminaries} \S \ref{overview of the proof of min-max theorem} \S \ref{An important vector field to deform the delta-mass} \S \ref{Constructing tightening vector field} and \S \ref{Construction of tightening map}), we give the proof for the min-max theory for Gaussian area.
\\

{\bf Acknowledgements}:  We would like to thank Toby Colding and Bill Minicozzi for their encouragement and several helpful conversations. D.K. is also grateful to Fernando Marques and Jacob Bernstein for their interest in this work. X. Z. would like to thank Detang Zhou for stimulating conversations which led to this project.


\section{Min-max theory for Gaussian area}\label{Min-max theory for Gaussian area}

In this section, we introduce the min-max theory for the Gaussian area. The min-max theory was originally developed by F. Almgren \cite{AF62, AF65} and J. Pitts \cite{P81} as a Morse-theoretical method for the purpose of constructing closed embedded minimal hypersurfaces in a closed Riemannian manifold. The heuristic idea behind Almgren-Pitts' work is to associate to every non-trivial 1-cycle in the space of hypersurfaces a critical point of the area functional, i.e. a minimal hypersurface.  One advantage of the Almgren-Pitts min-max theory is that it does not depend on the topology of the ambient space, and hence works in any closed manifold. This is especially useful when the ambient manifold does not support any stable minimal hypersurfaces, where minimization methods fail. The Simon-Smith theory (c.f. \cite{Sm82, CD}), is a later variant of the min-max theory specific to three manifolds and simplifies many of the complications in Almgren-Pitts caused by geometric measure theory, and also leads to a control of the genus for the minimal surfaces obtained (c.f. \cite{Sm82, DP10, K13}). Since self-shrinkers are unstable minimal surfaces for the Gaussian metric $\frac{1}{4\pi}e^{-\frac{|x|^2}{4}}dx^2$ in $\R^3$, any variational construction of self-shrinkers must be of the min-max type. In the following, we will develop a min-max theory for the Gaussian metric using the setup of the Simon-Smith theory (following the exposition of Colding-De Lellis \cite{CD}).

There are obvious difficulties to overcome as the ambient manifold considered here is non-compact and the metric behaves badly near infinity. To give a general idea of our strategy, let us first recall the four main ingredients in the Colding-De Lellis theory (based on the work of Simon-Smith \cite{Sm82}). The first is the so-called ''tightening" process, which is a pseudo-gradient flow of the area functional on the space of varifolds (generalized hypersurfaces, c.f. \cite[\S38]{Si83}). The second is a delicate combinatorical argument which leads to the existence of an ''almost minimizing" varifold.  The third ingredient is the regularity of almost minimizing varifolds.  The fourth ingredient is to control the genus of the almost minimizing varifold.  All the arguments related to the second, third and fourth ingredient are essentially local, so one can adapt them trivially to the Gaussian space. The first ingredient however is subtle here as the space of varifolds in $\R^3$ with bounded Gaussian area is no longer compact. In particular, a sequence of surfaces may weakly converge to a limit surface together with a point mass at infinity. To overcome this difficulty, we compactify $\R^3$ by adding a point at infinity to get the three sphere $S^3$, and view all the varifolds with bounded Gaussian area as varifolds on $S^3$. Then we make use of the special structure of the Gaussian metric to design a specific pseudo-gradient flow of the Gaussian area functional $F$ (\ref{Gaussian area}) in the space of varifolds on $S^3$. Our flow will either push a varifold to be $F$-stationary in $\R^3$, or decrease the mass near infinity. After applying this flow, all the sequence of surfaces of our interests will converge to a varifold stationary under the $F$ functional in $\R^3$ with no point mass at infinity, and hence fulfill our requirement.

Our work is the first instance of a global variational theory in a non-compact incomplete manifold. Instead of working by exhaustions as in Schoen-Yau's proof of the Positive Mass Conjecture \cite{SY79}, we work with the whole non-compact space and the surfaces therein. Before our work, R. Montezuma \cite{Mo14} developed a min-max theory in certain non-compact manifolds. Unfortunately, his method cannot be adapted to our setting. Firstly, our Gaussian space has a very bad end, which does not satisfy Montezuma's technical condition \cite[$*_k$-condition on page 1]{Mo14} near infinity. In addition, Montezuma's theorem essentially works in a compact manifold by cutting out the infinite end and eventually producing closed minimal surfaces.  In our case we need to work in the whole space, and the min-max surface we produce may in general be non-compact (see Example 2).

\vspace{2em}
Now we start with the setup. Our ambient manifold will be $\R^3$  equipped with the Gaussian metric $g^G=\frac{1}{4\pi}e^{-\frac{|x|^2}{4}}dx^2$.  A two dimensional surface $\Si$ in $\R^3$ (not necessarily compact), 
which is a critical point of the Gaussian area $F$ is called a {\em Gaussian minimal surface} or {\em self-shrinker}. Denote by $\Diff$ the identity component of the diffeomorphism group of $\R^3$. Let $\Js$ be the set of smooth isotopies, i.e. $\Js$ contains $\psi\in C^{\infty}([0, 1]\times \R^3, \R^3)$, such that $\psi(0, \cdot)=id$ and $\psi(t, \cdot)\in\Diff$ for all $t$.  Denote
\begin{itemize}
\vspace{-5pt}
\addtolength{\itemsep}{-0.7em}
\setlength{\itemindent}{0em}
\item $I^n=[0, 1]^n$ by the n-dimensional closed cube;

\item $\mathring{I}^n=(0, 1)^n$ by the n-dimensional open cube;

\item $\partial I^n=I^n\backslash \mathring{I}^n$.
\vspace{-5pt}
\end{itemize}
$I^n$ will be our parameter space in the following.  Denote by $\mathcal{Z}_2(\R^3)$ the space of surfaces in $\R^3$.
\begin{definition}\label{continuous family}
A family $\{\Si_\nu\}_{\nu\in I^n}$ of (smooth, two dimensional) surfaces in $\R^3$ is said to be a {\em continuous (genus $g$) family}, if
\begin{itemize}
\vspace{-5pt}
\addtolength{\itemsep}{-0.7em}
\setlength{\itemindent}{0em}
\item $\{\Si_{\nu}\}$ is a smooth family under the locally smooth topology;
\item For $t\in (0,1)^n$, $\{\Si_{\nu}\} $ is a genus $g$ surface
\item $\sup_{\nu\in I^n}F(\Si_{\nu})<\infty$;
\item $F(\Si_{\nu})$ is a continuous function of $\nu$.
\item $\{\Si_{\nu}\}_{\nu\in\partial I^n}$ contains only affine planes or empty sets.
\end{itemize}
\end{definition}

Given a continuous family $\{\Si_{\nu}\}_{\nu\in I^n}$, we can generate new continuous families by the following procedure. Denote $id: \R^3\rightarrow\R^3$ by the identity map. Take a map $\psi\in C^{\infty}(I^n\times \R^3, \R^3)$, such that $\psi(\nu, \cdot)\in\Diff$ for each $\nu\in I^n$. 
Define $\{\Si_{\nu}^{\pr}\}$ by $\Si^{\pr}_{\nu}=\psi(\nu, \Si_{\nu})$.
\begin{remark}
In general, $\{\Si^{\pr}_{\nu}\}_{\nu\in I^n}$ might not satisfy the requirement of Definition \ref{continuous family}. However, assuming that the set $\{\nu\in I^n: \psi(\nu, \cdot)\neq id\}$ is a compact subset of $\mathring{I}^n$, then in the following two cases which will be used later, $\{\Si_{\nu}^{\pr}\}$ does satisfy Definition \ref{continuous family}.
\begin{itemize}
\vspace{-5pt}
\addtolength{\itemsep}{-0.7em}
\setlength{\itemindent}{0em}
\item $\big(\psi(\nu, \cdot)-id\big)$ all have compact support, i.e. $\cup_{\nu\in I^n} spt\big(\psi(\nu, \cdot)-id\big)$ is a compact subset of $\R^3$;
\item $\psi(\nu, \cdot)$ is the time $1$ flow generated by smooth n-parameter family of vector fields $X_{\nu}: \R^3\rightarrow\R^3$, $\nu\in I^n$ with $\max_{\nu\in I^n}\|X_{\nu}\|_{C^1}\leq C$ (c.f. Lemma \ref{vector fields with C1 bound preserve continuous family}).
\end{itemize}
\end{remark}
A continuous family $\{\Si^{\pr}_{\nu}\}$ satisfying Definition \ref{continuous family} is said to be {\em homotopic to} $\{\Si_{\nu}\}$ if $\{\Si^{\pr}_{\nu}\}$ is gotten from $\{\Si_{\nu}\}$ under the above operation. A set of $\La$ of continuous families is a {\em saturated set} of $\{\Si_{\nu}\}$ (or a {\em homotopic class} of $\{\Si_{\nu}\}$) if any $\{\Si^{\pr}_{\nu}\}\in\La$ is homotopic to $\{\Si_{\nu}\}$.
\begin{remark}
By our definition,  all our continuous families agree on $\partial I^n$.
\end{remark}

Given a family $\{\Si_{\nu}\}\in\La$, denote $\F(\{\Si_{\nu}\})$ by the maximal Gaussian area of its slices,
\begin{equation}
\F(\{\Si_{\nu}\})=\max_{\nu\in I^n}F(\Si_{\nu}).
\end{equation}
The {\em min-max value}, denoted by $W(\La)$, is the infimum of $\F$ taken over all families in $\La$,
\begin{equation}
W(\La)=\inf_{\{\Si_{\nu}\}\in \La}\big[\max_{\nu\in I^n}F(\Si_{\nu})\big].
\end{equation}
If a sequence $\{\{\Si_{\nu}\}^k\}\subset \La$ satisfy $\lim_{k\rightarrow\infty}\F(\{\Si_{\nu}\}^k)=W(\La)$, we say the sequence a {\em minimizing sequence}. Let $\{\{\Si_{\nu}\}^k\}$ be a minimizing sequence and $\{\nu_k\}$ a sequence of parameters. If $\lim_{k\rightarrow\infty}F(\Si^k_{\nu_k})=W(\La)$, then we say $\{\Si^k_{\nu_k}\}$ a {\em min-max sequence}.

The main result we need (and proved in Part \ref{min-max for Gaussian area}) is the following:
\begin{theorem}\label{main min-max theorem}
For any saturated set $\La$ of genus $g$ families, if $W(\La)>\max_{\nu\in\partial I^n}F(\Si_{\nu})$, then there is a min-max sequence of $\La$ converging in the sense of varifolds to a connected, smooth, embedded, Gaussian minimal surface with Gaussian area $W(\La)$ (counted with multiplicity) and with genus at most $g$.
\end{theorem}
\noindent
Let us first consider several instances of Theorem \ref{main min-max theorem} which will be useful later.  Denote by $\mathcal{Z}_2(\mathbb{R}^3)$ the set of closed embedded but possibly trivial surfaces in $\mathbb{R}^3$.
\\
\\
\noindent
\textbf{Example 1:} Consider the one parameter sweepout of $\mathbb{R}^3$ by parallel affine planes: $$\Phi:[0,1]\rightarrow\mathcal{Z}_2(\mathbb{R}^3)$$ defined by $$\Phi(t):=\{(x,y,z)\in\mathbb{R}^3\:|\: z=\tan(\pi (t-1/2)\}$$  Of course $\Phi(0)=0$ and $\Phi(1)=0$.  Denote by $\Lambda_\Phi$ the collection of sweepouts that is a saturation of this sweepout.  We claim 
\begin{equation}\label{or}
W(\Lambda_{\Phi})=1.
\end{equation}  Moreover the width is achieved by the self-shrinking plane $\{z=0\}$ and the sweepout $\Phi$ is therefore optimal. To see \eqref{or}, observe that  by the definition of width we have $W(\Lambda_\Phi)\leq 1$.  Moreover, if $0<W(\Lambda_\Phi)<1$, then the Min-Max Theorem would produce a self-shrinker with entropy smaller than $1$,  an impossibility.  So it remains to show $0<W(\Lambda_\Phi)$.  We can rule this out using the isoperimetric inequality in Gaussian space, which says that affine planes are the isoperimetric surfaces:
\begin{lemma}(Isoperimetric Theorem in Gaussian Space \cite{B},\cite{ST})\label{iso}
For a Borel set in $\mathbb{R}^n$ of Gaussian volume $V$, $\partial V$ has Gaussian measure at least that of the affine plane bounding a volume $V$.
\end{lemma}
\noindent
Any element of the homotopy class $\Lambda_{\Phi}$ must contain a surface $\Sigma$ that bounds a region $R$ of volume $1/2$.  By the Isoperimetric Theorem \ref{iso}, we get $F(\Sigma)\geq 1$.  Thus we have shown \eqref{or}.
\\
\\
The folllowing example will also be useful later on:
\\
\\
\noindent
\textbf{Example 2:} Consider the sweepout of $\mathbb{R}^3$ given by $2$-spheres:
$$\Phi(t):=\{(x,y,z)\in\mathbb{R}^3\:|\:x^2+y^2+z^2 = \tan(\pi t/2)\}.$$
Let $\Lambda_{\Phi}$ be the homotopy class generated by this sweepout. We want to show that in this case too we have $W(\Lambda_\Phi)=1$ and thus the sweepout by spheres, although it contains a self-shrinker, is inefficient. One way to prove this would be to argue that the min-max limit produced has Morse index $1$, and so must be a plane (where the one negative eigenfunction is translation normal to the plane).  Instead we will argue as above.  Indeed, by the isoperimetric argument we know that $0<W(\Lambda_\Phi)$.   So again $W(\Lambda_\Phi)\geq 1$.  To see that equality is achieved, consider for $\tau>0$ the new translated and dilated family $\Phi'(t)=\tau(\Phi(t)-(1,0,0))$ . As $\tau$ approaches infinity, in any fixed ball about the origin, the sweepout surfaces $\Phi'(t)$ converge to a foliation by affine planes, and thus the maximal area of a slice approaches $1$.  This confirms that $W(\Lambda_\Phi)=1$.
\begin{remark}
Note that the optimal sweepout (Example 1) for the saturated family of Example 2 is not contained in the saturation of Example 2.  
\end{remark}

\part{Entropy bound}\label{entropy bound}

\section{Canonical Family}\label{canonical family}
Given a smooth closed embedded surface $\Sigma$ of genus $g$ in $\mathbb{R}^3$ we will construct a $4$-parameter continuous family of surfaces associated to it. 
 For any $t\in\mathbb{R}^3$ and $s\in\mathbb{R}_{\geq 0}$ we define a surface:
\begin{equation}
\Sigma_{t,s}:= s(\Sigma-t).
\end{equation}
The surface $\Sigma_{t,s}$ corresponds to translating the point $t$ to the origin and dilating by a factor $s$.   Thus associated to $\Sigma$ we have a $4$-parameter sweepout:
\begin{equation}
\Pi_{\Sigma}:\mathbb{R}^3\times\mathbb{R}_{\geq 0}\rightarrow\mathcal{Z}_2(\mathbb{R}^3)
\end{equation}
\begin{equation}
\Pi_{\Sigma}(t,s):=\Sigma_{t,s}
\end{equation}
We want to first understand our sweepout as we approach the boundary of $\mathbb{R}^3\times\mathbb{R}_{\geq 0}$.   For any $t\in\Sigma$, we have
\begin{equation}
\lim_{s\rightarrow\infty}\Sigma_{t,s}=T_t\Sigma(0).
\end{equation}
Here $T_t\Sigma(0)$ denotes the tangent plane of $\Sigma$ translated to pass through the origin.
For any $t\in\R^3\setminus\Sigma$ we have:
\begin{equation}
\lim_{s\rightarrow\infty}\Sigma_{t,s}=0.
\end{equation}
Also for any $t$, 
\begin{equation}
\lim_{s\rightarrow 0}\Sigma_{t,s}=0.
\end{equation}
Also for any fixed $s$, by the compactness of $\Sigma$ we have
\begin{equation}
\lim_{|t|\rightarrow\infty}\Sigma_{t,s}=0.
\end{equation}
All of the surfaces in our sweepout have the same genus and vary smoothly, though toward the boundary (in parameter space) they begin to degenerate to points or planes.  Notice also that for our initial sweepout, since entropy is the supremum over all centers and scales, the entropy controls all the Gaussian areas of the sweepout:
\begin{equation}\label{areabound}
\sup_{t,s} F(\Sigma_{t,s})\leq\lambda(\Sigma).
\end{equation}
This inequality is fundamental in the proof of Theorem \ref{main} (and is analogous to the fact that the Willmore energy controls the areas of the canonical five parameter sweepout discovered by Marques-Neves).

From $\Sigma$ we have constructed a sweepout $\Pi_{\Sigma}$ where at the "boundary" of $\mathbb{R}^3\times\mathbb{R}_{\geq 0}$ the sweepout  is either an oriented plane through the origin or the zero surface.  We get a plane precisely when restricted to $\Sigma\times\{\infty\}$.  This plane coincides with the Gauss map of $\Sigma$.  Precisely, 
\begin{equation}
\Pi_\Sigma(\cdot ,\infty):\Sigma\rightarrow\mathbb{S}^2 \text{  is the Gauss map.}
\end{equation}
\noindent
The degree of the Gauss map is $1-g$.  Thus if $\Sigma$ is not diffeomorphic to a two-torus, then the Gauss map has non-zero degree.  Just as in the proof of the Willmore conjecture, it is extremely important that our sweepout at the boundary encodes the geometry of the surface $\Sigma$. This will be essential in the degree argument to rule out our min-max limit becoming a plane.  

\subsection{Boundary blow-up}
At first glance, our sweepout seems to be discontinuous at the top face $\mathbb{R}^3\times\{\infty\}$ because it consists of either planes or the zero surface.  It turns out however that depending on the angle of approach to a point in  $\Sigma\times\{\infty\}$ one actually gets all the affine planes extrapolating between these two extremes.  We must do a blow-up argument as in the proof of the Willmore Conjecture to capture all the limits and make our canonical family continuous.  In the end we will produce a new family from the original one where in a tubular neighborhood around $\Sigma\times\{\infty\}$ in parameter space, the surfaces die off via affine planes approaching infinity (in Gaussian space these planes have area approaching zero).  This will give a good canonical family to which the Min-Max theorem can be applied.  
\\
\indent
To perform the blowup, we first observe that if there is a balance between the scaling factor and distance to $t\in\Sigma$ as the parameters approach the boundary of the sweepout, and we approach $t$ at a constant angle, then we get an affine plane:
\begin{lemma}\label{silly}
If  $(a_i,s_i)\rightarrow (t,\infty)$ with $t\in\Sigma$ and for some $C>0$ we have  $|a_i-t|s_i\rightarrow C$, and  $a_i-t=-|a_i-t|v_i$  where $v_i\rightarrow v$ for some $v\in\mathbb{S}^2$ then
\begin{equation}
\lim_{i\rightarrow\infty}\Sigma_{a_i,s_i}=T_t(\Sigma)(0)+Cv
\end{equation}
\end{lemma}

\noindent
\emph{Proof:}  We simply compute 
\begin{equation}
\lim_{i\rightarrow\infty} s_i(\Sigma-a_i)=\lim_{i\rightarrow\infty}s_i(\Sigma-t+|a_i-t|v_i)=T_t\Sigma(0)+\lim_{i\rightarrow\infty}s_i(|a_i-t|v_i)=T_t\Sigma(0) + Cv
\end{equation}
\qed
\\
\indent
We will now explain how to do the blowup at the boundary of our sweepout.  It is convenient first to reparameterize the parameter space of the scaling parameter by an explicit homeomorphism:
\begin{equation}
h:[0,\infty)\rightarrow [0,1) \mbox{ where   }  h(t)=(2/{\pi})\tan^{-1}(t).
\end{equation}
Thus our sweepout $\Pi_\Sigma$ can now be considered as maps from $\mathbb{R}^3\times [0,1)$ to $\mathcal{Z}_2(\mathbb{R}^3)$ by setting $\Pi_\Sigma(p,t):=\Pi_\Sigma(p,h^{-1}(t))$.  The goal of this section is then to extend our sweepout $\Pi_{\Sigma}$ continuously to $\mathbb{R}^3\times [0,1]$.  

First let us endow $\mathbb{R}^3\times [0,1]$ with the product metric and consider $\Omega_{\epsilon}$ defined to be the $\epsilon$-tubular neighborhood of $\Sigma\times\{1\}$ inside $\mathbb{R}^3\times [0,1]$.  In other words, 
$$\Omega_{\epsilon}=\{x\in\mathbb{R}^3\times [0,1]: |x-(p,1)| < \epsilon \mbox{ for some } p\in\Sigma\}.$$   
The following lemma explains what the end result of our  blow-up procedure will be.  The sweepout is untouched away from $\Sigma\times\{1\}$ but in a neighborhood $\Omega_\epsilon$ of $\Sigma\times\{1\}$ it consists of affine planes that converge to planes as points in $\Omega_\epsilon$ converge to $\Sigma\times\{1\}$.  Also it is crucial that we keeep the Gaussian areas in our sweepout still bounded by the entropy $\lambda(\Sigma)$.  In the following lemma, let $R$ denote the nearest point projection onto $\Sigma$ in $\mathbb{R}^3$.  Denote by $T_\epsilon(\Sigma)$ the tubular neighborhood about $\Sigma$ in $\mathbb{R}^3$.  For suitably small $\epsilon>0$, $R$ maps $T_\epsilon(\Sigma)$ to $\Sigma$.

\begin{lemma}\label{boundary blow-up lemma}(Boundary Blow-Up Lemma)
Given $$\Pi_{\Sigma}:\mathbb{R}^3\times [0,1)\rightarrow\mathcal{Z}_2(\mathbb{R}^3)$$ that is obtained from reparameterizing the canonical family associated to $\Sigma$ as above, for any $\epsilon>0$ small enough we can produce a new sweepout  $$\bar{\Pi}_{\Sigma}:\mathbb{R}^3\times [0,1]\rightarrow\mathcal{Z}_2(\mathbb{R}^3)$$ such that the following properties hold:
\begin{itemize}
\addtolength{\itemsep}{-0.5em}
\setlength{\itemindent}{0em}
\item $\{\bar{\Pi}_\Sigma(x): x\in\R^3\times[0, 1]\}$ is a continuous family in the sense of Definition \ref{continuous family},
\item $\bar{\Pi}_\Sigma(x) = \Pi_\Sigma(x)$ for $x\in\mathbb{R}^3\times[0,1)\setminus\Omega_{2\epsilon}$,
\item For $x\in\bar{\Omega}_\epsilon$, $\bar{\Pi}_\Sigma(x)$ is an affine plane,

\item For $(p,1)\in\Sigma\times\{1\}$, $\bar{\Pi}_\Sigma(p,1)$ is $T_p(\Sigma)(0)$ (i.e. the tangent plane $T_p(\Sigma)$ translated to the origin),
\item For $p\in\Sigma$, the surfaces associated to the line segment $R^{-1}(p)\times\{1\}$ restricted to $T_\epsilon(\Sigma)\times\{1\}$ consist of all affine planes parallel to $T_p\Sigma(0)$,
\item $\sup_{x\in\mathbb{R}^3\times [0,1]} F(\bar{\Pi}_\Sigma(x))\leq\lambda(\Sigma)$.
\end{itemize}
\end{lemma}
\noindent
\emph{Proof:}
First we will parameterize $\Omega_{2\epsilon}$ sitting inside $\mathbb{R}^3\times [0,1]$.  The set $\Omega_{2\epsilon}$ is diffeomorphic to $\Sigma\times D_+$ where $D_+$ is a half-disk. Fix $p\in\Sigma$.  Our parameters $(\tau, \rho)$ live on a small half disk $D_+$ of radius $2\epsilon$  in $\mathbb{R}^2$ centered around the point $(1,0)$ in the region in $\mathbb{R}^2$ where $\tau<1$.  The $\rho$ parameter ($-2\epsilon\leq\rho\leq 2\epsilon$) is the signed distance in $\mathbb{R}^3$ from a point to $p$ (where the negative values are taken outside $\Sigma$).  The $\tau$ parameter is the scaling parameter in the vicinity of $1$ (so $1-2\epsilon\leq\tau\leq 1$.  For given $p\in\Sigma$, denote $\Sigma_{\tau, \rho}=\bar{\Pi}_{\Sigma}(p+\rho n(p), \tau)$ (where $n(p)$ is the outward-pointing unit normal vector in $\mathbb{R}^3$ to $\Sigma$ at $p$). By Lemma \ref{silly} we see that if for some constant $C$,  the equality $\rho(t)\tan(\frac{\pi}{2}\tau(t))=C$ holds as  $\rho(t)\rightarrow 0$ and $\tau(t)\rightarrow 1$ we get a well-defined limit surface in our sweepout: 
$$\lim_{\tau\rightarrow 1, \rho\rightarrow 0} \Sigma_{\tau, \rho}= T_p(\Sigma)(0)+Cn(p).$$   Thus by varying $C$ among all real numbers we obtain a family of functions $\rho_C(\tau)$ given by $\rho_C(\tau)=C cot(\frac{\pi}{2}\tau)$ that foliates the set $D_+$ and so that all members of this family pass through two points on the boundary of the half-disk: they each pass through the point $(1,0)$ and through one (varying in $C$) point along the curved half-circle part of the boundary of $D_+$.  In other words, approaching $(p,1)$ along any such curve gives us a unique limit surface and the union of such curves foliates the parameter space $D_+$. \\
\indent
Now we can follow Marques-Neves \cite{MN12} rather closely to do the blowup which captures this one-parameter family of limits.  We first construct a blowup map $B$ which is a continuous map (not one-to-one though) from $\mathbb{R}^3\times [0,1]$ to itself which takes $\mathbb{R}^3\times [0,1]\setminus\Omega_{\epsilon}$ to all of $\mathbb{R}^3\times [0,1]$. The map is constructed so that if a point $x$ hits $\partial\Omega_{\epsilon}$ at $(\tau(p), \rho(p))$ then $B(x)$ maps to the limit achieved as we approach $\Sigma$ along the direction given by the unique foliating curve intersecting $(\tau(p), \rho(p))$ in $D_+$.  Precisely we construct a continuous map $B:\mathbb{R}^3\times [0,1]\rightarrow\mathbb{R}^3\times [0,1]$ such that
\begin{itemize}
\addtolength{\itemsep}{-0.5em}
\setlength{\itemindent}{0em}
\item $B$ is the identity map in $(\mathbb{R}^3\times [0,1))\setminus\bar{\Omega}_{2\epsilon}$
\item $B$ maps $\Omega_{2\epsilon}\setminus\bar{\Omega}_\epsilon$ homeomorphically onto $\Omega_{2\epsilon}$ 
by reparameterizing the curves $\rho_C(\tau)$.

\end{itemize}
We then define a new family $\bar{\Pi}_\Sigma$ defined on all of $\mathbb{R}^3\times [0,1]$ so that for $v\in\mathbb{R}^3\times [0,1)\setminus\bar{\Omega}_\epsilon$, we set
\begin{equation}
\bar{\Pi}_\Sigma(v) = \Pi_{\Sigma}(B(v)).
\end{equation}
We then further extend $\bar{\Pi}_\Sigma$ in the region $\Omega_\epsilon$ to be constant along the line segments starting at any point of the form $v=(p,y)\in\partial\Omega_{\epsilon}$  and ending at $(p,1)$.

Our new canonical family is now extended to all of $\mathbb{R}^3\times [0,1]$ and for $p\in\Sigma$ we still have that $\bar{\Pi}_\Sigma(p,1)= T_p(\Sigma)(0)$. Also it is easy to see that our new family $\bar{\Pi}_\Sigma$ is a continuous family in the sense of Definition \ref{continuous family} (although the parameter space is non-compact, we can still compactify it in a straightforward way). If we consider a tubular neighborhood of radius $\epsilon$ around $\Sigma$ (in $\mathbb{R}^3$) and pick a point $p\in\Sigma$ and move $x$ from $p$ normally in $\mathbb{R}^3$, then $\bar{\Pi}_{\Sigma}(x,1)$ varies through all possible affine planes parallel to $T_p(\Sigma)(0)$.\qed
\indent


\section{Min-max argument: Proof of the Main Theorem}\label{min-max argument and proof of the main theorem}
\noindent
\emph{Proof of Theorem \ref{main}:   }
Let $\Sigma$ be a closed two-sphere in $\R^3$.  We can first of all assume that $\lambda(\Sigma)<3/2$ (otherwise the Main Theorem is trivial since $3/2\geq\lambda(\mathbb{S}^2)$).
We will run a min-max argument for the Gaussian area functional on all sweepouts of $\mathbb{R}^3$ in the saturation of our initial sweepout $\bar{\Pi}_\Sigma$ (the canonical family made continuous in the previous section).  Denote by $\Lambda_{\Sigma}$ the collection of all sweepouts obtained by saturating $\bar{\Pi}_\Sigma$.  It follows from the definition of width that the width of $\Lambda_\Sigma$ is at most $\lambda(\Sigma)$.  Since the areas of the surface on the boundary of sweepouts in $\Lambda_\Sigma$ are at most 1, if we can show 
\begin{equation}\label{lowerboundwidth}
W(\Lambda_\Sigma)>1,
\end{equation} then the Min-Max Theorem for Gaussian Area \ref{main min-max theorem} will produce a self-shrinker, $\tilde{\Sigma}$  realizing the width.  So by \eqref{areabound} we obtain
\begin{equation}\label{2}
1<F(\tilde{\Sigma})\leq\lambda(\Sigma)<3/2.
\end{equation}
But for a self-shrinker, the entropy is realized by the Gaussian area so in fact \eqref{2} implies
\begin{equation}\label{done}
1<\lambda(\tilde{\Sigma})\leq\lambda(\Sigma)<3/2.
\end{equation}
 By Theorem \ref{main min-max theorem}, the genus of $\tilde{\Sigma}$ is $0$.  Thus by Brendle \cite{BR}, $\tilde{\Sigma}$ is either a sphere, cylinder or plane.   In each case the multiplicity must be one (by the bound of $3/2$ on $\lambda(\Sigma)$).  The cylinder is ruled out because it has entropy bigger than $3/2$.  The plane is ruled out because its entropy is $1$.  Thus we must have $\tilde{\Sigma}=\mathbb{S}^2$ which yields the inequality $\lambda(\Sigma)\geq\lambda(\mathbb{S}^2),$ and we are done.  It remains to show \eqref{lowerboundwidth}.  

\begin{proposition}\label{big}
If the genus of $\Sigma$ is not $1$, then $W(\Lambda_{\Sigma}) > 1$.
\end{proposition}
\begin{remark}
Since Proposition \ref{big} holds for all closed surfaces with genus not equal to $1$, one can apply the argument of Theorem \ref{main} to any closed surface with genus not $1$. If one knew that the shrinker $\tilde{\Sigma}$ that arises in the proof had index at most $4$ (since it arises from a $4$ parameter sweepout) one can use Colding-Minicozzi's classification of such shrinkers instead of Brendle's result for genus $0$, to obtain $\lambda(\Sigma)\geq\lambda(\mathbb{S}^2)$ for such surfaces.  Recently F. Marques and A. Neves have announced upper index bounds for min-max limits in the smooth setting.  It is very likely their argument applies in the Gaussian min-max setting too.  In their proof of the Willmore conjecture, Marques-Neves \cite{MN12} get around the issue of index bounds for their 5 parameter sweepout by first considering the minimal surface of {\it smallest area} and its canonical family.  Unfortunately this trick does not work here as the smallest area self-shrinker may well be non-compact.  It is not so clear how to adapt the degree argument when the surface is non-compact.
\end{remark}
\begin{remark}
It is a curious fact that the degree argument fails when the genus is $1$.  In the Willmore conjecture, the degree argument fails in the case where the genus is zero -- however in that setting all (even immersed) minimal two-spheres were classified earlier by F. Almgren \cite{AF66} and known to be the equatorial two-spheres.  There is as of yet however no such classification for genus $1$ self-shrinkers.  We know so far only of the rotationally symmetric Angenent torus and the genus $1$ surface with two ends discovered numerically by Chopp \cite{C} and described further in Ilmanen \cite{I}.
\end{remark}
\noindent
\subsection{Proof of Proposition \ref{big}}

Denote by $H$ the component of $\mathbb{R}^3\setminus\Sigma$ which is an open handlebody.  We will argue by contradiction and show that if Proposition \ref{big} were false, the Gauss map $G$ defined on $\Sigma$ would extend to a continuous map defined on $H$.   But the Gauss map on a genus $g$ surface has degree $1-g$, and it follows from basic topology that no such extension of $G$ to the handlebody $H$ can exist if $g\neq 1$.  We shall need the following extension of this, where we consider our Gauss map as a map into the projective plane $\mathbb{RP}^2$ rather than $\mathbb{S}^2$:
\begin{lemma}\label{rp2}
Let $H$ be a closed handlebody in $\mathbb{R}^3$ with boundary a surface $\Sigma$ of genus $g$.  For $g\neq 1$, the (reduced) Gauss map $\tilde{G}:\Sigma\rightarrow\mathbb{RP}^2$ cannot extend to a continuous map defined on all of $H$.
\end{lemma}
\noindent
\emph{Proof:}  Since the map $\tilde{G}:\Sigma\rightarrow\mathbb{RP}^2$ factors through $\mathbb{S}^2$ and $\pi_1(\mathbb{S}^2)$ is trivial, the induced map  $\tilde{G}_*:\pi_1(\Sigma)\rightarrow\pi_1(\mathbb{RP}^2)$ is trivial.  There also exists a natural surjective map $i_*:\pi_1(\Sigma)\rightarrow\pi_1(H)$.  
 If  there existed a map $E:H\rightarrow\mathbb{RP}^2$ which is an extension of $\tilde{G}$ to $H$, then $\tilde{G}_*=E_*\circ i_*=0$, and since $i_*$ is surjective, this implies $E_*=0$.  Since the induced map $E_*$ on $\pi_1$ from $H$ to $\mathbb{RP}^2$ vanishes, this means that $E$ lifts to a map from $H$ to $\mathbb{S}^2$ agreeing with $\tilde{G}$ on $\Sigma$.  But this is impossible.\qed
\\

We will also need the following simple observation in the degree argument. If we consider rescalings about points "inside" of a genus g handlebody then we sweep out all of the ambient space:
\begin{lemma}\label{sweep} For any fixed $t\in H$, the one-parameter sweepout by dilates of $\Sigma$, $\{s(\Sigma-t)\}_{s\in[0,1]}$ sweeps out all of $\mathbb{R}^3$ (i.e. the width of the homotopy class of any saturation of such sweepouts is greater than $0$).
\end{lemma}
\begin{proof}
This is a straightforward consequence of the isoperimetric argument from Example 1.
\end{proof}
\noindent
\emph{Proof of Proposition \ref{big}:}   We first explain the idea.  We argue by contradiction, so assume $W(\Lambda_\Sigma)=1$.  This implies that we have a sequence of sweepouts $\Phi_i$ in the saturated family $\Lambda_\Sigma$ with maximal Gaussian areas approaching 1 from above.  For any $t\in H$, the one parameter family of surfaces (in $s$) $[0,1]\rightarrow\Phi_i(t,s)$ also sweeps out $\mathbb{R}^3$ by Lemma \ref{sweep} and since the maximal areas are approaching $1$, there must be some $s$ so that $\Phi_i(t,s)$ is very close to a signed plane.  Thus for each $t\in H$ we essentially can produce a signed plane.  The choice of plane will depend continuously on $t$ thus giving us a continuous map from $H$ to $\mathbb{S}^2$ that extends the Gauss map $G$ defined on the boundary.   This is impossible when $g\neq 1$ by Lemma \ref{rp2}. \\
\indent We now give the detailed argument.  Assume toward a contradiction that we have a sequence of sweepouts $\Phi_i$ with maximal Gaussian areas approaching $1$. Moreover, we can assume that $\{\Phi_i\}$ is a tightened sequence as in Lemma \ref{pull tight}.  Denote by $\delta$ the small parameter used in the boundary blowup argument.  Note that the surfaces associated to parameters in $\R^3\times\{0\}$ and $(\R^3\setminus T_\delta(\Sigma))\times\{1\}$ are the trivial surfaces.  Denote by $\Sigma^+_\epsilon$ the component of the boundary of the tubular neighborhood of radius $\epsilon$ about $\Sigma$ in $\R^3$ that is inside $H$, and by $\Sigma^-_\epsilon$ the other component.  Also denote by $H^+_\epsilon$ the handlebody bounded by $\Sigma^+_{\epsilon}$ and let $H^-_\epsilon$ be $\R^3$ minus the handlebody bounded by $\Sigma^-_\epsilon$. It is convenient in the following to put a rotationally symmetric metric $dist$ on the space of $2$-varifolds in $\R^3$ with bounded Gaussian mass (the precise such metric will be constructed in Section \ref{notation}, c.f. the $\bF_{S^3}$ metric on $\V_2^G(\R^3)$). Denote by $\mathcal{P}$ the space of planes in $\R^3$.  \\
\indent  For each $i$ let $\bar{A}(i)$ be the set of $x\in \R^3\times [0,1]$ with the property that $dist(\Phi_i(x),\mathcal{P})\geq\epsilon$.  Note that for $\epsilon$ sufficiently small, $\bar{A}(i)$ contains $\R^3\times\{0\}$ and intersects the top face precisely in $(H^+_{\epsilon'}\cup H^-_{\epsilon'}) \times\{1\}$ for some $\epsilon'=\epsilon'(\epsilon)$ (where $\epsilon'\rightarrow 0$ as $\epsilon\rightarrow 0$). Fix $\epsilon$ so small that $\epsilon'\leq\delta$, where $\delta$ is again the parameter used in the boundary blowup argument.  Let $A(i)\subset\bar{A}(i)$ be the connected component containing the part of the top face containing $H^+_{\epsilon'}\times\{1\}$.  \\
\indent Our first claim is that for $i$ large enough, $A(i)$ can never intersect the bottom face $\R^3\times\{0\}$.   Indeed, if the claim were false, we would have a sequence of continuous paths $\gamma_i:[0,1]\rightarrow A(i)$ that begin in $H^+_{\epsilon'}\times\{1\}$ and end in $\R^3\times\{0\}$.  There are two cases to consider, depending on whether $\gamma_i$ begins in $H^+_\delta$ or $H^+_{\epsilon'}\setminus H^+_\delta$.  In the first case, since we are assuming $$\lim_{i\rightarrow\infty}\sup_x F(\Phi_i(x))=1$$ and since by Lemma \ref{sweep} each path $\gamma_i$ sweeps out  $\mathbb{R}^3$, we know that $\{\ga_i\}$ is also a tightened sequence of the homotopy class described in Example 1 of \S \ref{Min-max theory for Gaussian area} in the sense of Lemma \ref{pull tight}. Applying Proposition \ref{existence of a.m. sequence} and Theorem \ref{regularity for a.m. surface} and the discussion in Example 1 of \S \ref{Min-max theory for Gaussian area} leads to the fact that for $i$ large there must exist a $t_i$ so that $\Phi(\gamma(t_i))$ is within $\epsilon/2$ of the space of planes $\mathcal{P}$, contradicting the definition of $A(i)$.  In the second case where $\gamma_i$ begins in $(H^+_{\epsilon'}\setminus H^+_\delta)\times\{1\}$, the surfaces do not sweep-out all of $\R^3$ but rather the side of a half-space of larger area, and thus some slice must also approach a self-shrinking plane in this case as well.  Thus the claim is established.  For any fixed $i$, it also cannot happen that the projection of  $A(i)$ onto the $\R^3$ factor is unbounded, because as $|t|\rightarrow\infty$, the surfaces in our sweepout must converge to trivial surfaces in order to be in the saturation $\Lambda_\Sigma$ of the canonical family $\bar{\Pi}_\Sigma$.

Therefore, since $A(i)$ only intersects the top face, we can approximate $\partial A(i)$ by a smooth handlebody $H'(i)\subset\R^3\times [0,1]$ with boundary $\Sigma^+_{\epsilon'}\times\{1\}$ so that all surfaces associated to $H'(i)$ (for $i$ large) are within $2\epsilon$ and $\epsilon/2$ of $\mathcal{P}$ in the metric $dist$.    We now construct the desired extension of $\tilde{G}$ to $H'(i)$.  Let $\eta$ be chosen so that all balls of radius $\eta$ in $\mathbb{RP}^2$ are geodesically convex (i.e., there exists a unique and minimizing geodesic between any two points).  First observe (c.f. Section 9.10 in \cite{MN12}) that we can find a $C>0$ so that
\begin{equation}\label{cont}
\mbox{If } P_1,P_2\in\mathcal{P}\mbox{  satisfy  } dist(P_1,P_2) < C\eta, \mbox {  then  }  dist_{\mathbb{RP}^2}(P_1,P_2)\leq\eta.
\end{equation}

Given two surfaces, $S_1$ and $S_2$ satisfying $dist(S_1,S_2)\leq\epsilon$ and $dist(S_1,\mathcal{P})\leq 2\epsilon$ and \\$dist(S_2,\mathcal{P})\leq 2\epsilon$, it follows by the triangle inequality that any choice of nearest point projections in $\mathcal{P}$, $P_1$ for $S_1$ and $P_2$ for $S_2$ are within $5\epsilon$ of each other in the metric $dist$.  Let $\epsilon$ now be chosen so small so that $5\epsilon\leq C\eta$. We obtain from \eqref{cont} that $dist_{\mathbb{RP}^2}(P_1,P_2)\leq \eta$.  Thus while nearby surfaces in our sweepout may have multiple nearest point projections to $\mathcal{P}$, they can all be chosen to lie in a geodesically convex neighborhood in $\mathbb{RP}^2$.  

We now explain how to build the map $\tilde{G}$ from $H'(i)$ to $\mathbb{RP}^2$ extending the Gauss map.  Recall that every surface corresponding to points in $H'(i)$ lies in a $2\epsilon$ neighborhood of $\mathcal{P}$ (for $i$ large enough).  By continuity we can triangulate the handlebody $H'(i)$ so finely so that the surfaces corresponding to any two adjacent vertices of the triangulation are within $\epsilon$ in the metric $dist$.  For each vertex in the triangulation, define $\tilde{G}$ to be some choice of nearest point projection to $\mathcal{P}$.  For vertices in the interior of the triangulation, it does not matter which nearest point projection one chooses.  For vertices at the boundary points $x\times\{1\}$ of $H'(i)$, choose the point in $\mathbb{RP}^2$ obtained by first retracting $x\in \Sigma^+_{\epsilon'}$ to $R(x)\in \Sigma$  in $\mathbb{R}^3$ via nearest point projection and setting $\tilde{G}$ at $x\times\{1\}$ to be the Gauss map of $\Sigma$ at $R(x)$.

By construction, for any two vertices $v_1$ and $v_2$ in $H'(i)$, the corresponding planes $\tilde{G}(v_1)$ and $\tilde{G}(v_2)$ are contained in a geodesically convex ball in $\mathbb{RP}^2$.  We can thus extend $\tilde{G}$ along the edge $e_{12}$ in $H'(i)$ joining $v_1$ and $v_2$ via the unique minimizing geodesic in $\mathbb{RP}^2$ joining $v_1$ and $v_2$ (contained in this geodesically convex ball).    To extend to the $2$-skeleton, observe that for any three adjacent vertices $v_1$, $v_2$ and $v_3$, we have that  $\tilde{G}(v_2)$, and $\tilde{G}(v_3)$ are both within $\eta$ of $\tilde{G}(v_1)$.  Thus we can extend $\tilde{G}$ to the face $\mathcal{F}$ in $H'(i)$ determined by $v_1$, $v_2$ and $v_3$ by the interior of the corresponding triangle in the geodesic ball of $\mathbb{RP}^2$.  The same process can be repeated over the $3$-cells.  Thus we obtain iteratively a map $\tilde{G}:H'(i)\rightarrow\mathbb{RP}^2$ extending the Gauss map, contradicting Lemma \ref{rp2}.\qed


\part{Min-max for Gaussian area}\label{min-max for Gaussian area}


In the remaining sections, we prove the Min-max Theorem \ref{main min-max theorem} for the Gaussian Area, which has been the main ingredient in our arguments.

\begin{proof}
It is a direct corollary of Proposition \ref{pull tight}, Proposition \ref{existence of a.m. sequence}, Theorem \ref{regularity for a.m. surface} and Remark \ref{remark for regularity for a.m. surface}.
\end{proof}



\section{Preliminaries}\label{preliminaries}

\subsection{Notation}\label{notation}
We first list a few notations and definitions used in the following. For concepts in geometric measure theory, we mainly refer to \cite{Si83}.
\begin{itemize}
\vspace{-5pt}
\addtolength{\itemsep}{-0.7em}
\setlength{\itemindent}{0em}

\item $S^3$ and $\R^3$ denote the 3-dimensional standard sphere and Euclidean space respectively. Sometime we will view $S^3$ as the one point compactification $\R^3\cup\{\infty\}$. We will also identify $\R^3$ with $S^3\backslash\{\infty\}$.

\item $ds_0^2$ and $dx^2$ denote the round metric on $S^3$ and Euclidean metric on $\R^3$ respectively.

\item $G_2(S^3)$ and $G_2(\R^3)$ denote the Grassmannian bundle of un-oriented 2-planes over the tangent bundle $TS^3$ of $S^3$ or $T\R^3$ of $\R^3$ respectively (c.f. \cite[\S 38]{Si83}).

\item $Lip\big(G_2(S^3)\big)$ denotes the space of Lipschitz functions on $G_2(S^3)$ with respect to the induced metric by $ds_0^2$.

\item $\V_2(S^3)$ denotes the space of 2-varifolds on $S^3$, i.e. Radon measures on $G_2(S^3)$.

\item $(\R^3, g^G)$ denotes the Gaussian space, where $g^G=\frac{1}{4\pi}e^{-\frac{|x|^2}{4}}dx^2$.

\item $\X(\R^3)$ denotes the space of vector fields in $\R^3$. $\X_c(\R^3)$ denotes the subset of vector fields in $\X(\R^3)$ with compact support.

\item $\V_2^G(\R^3)$, or equivalently $\V_2(\R^3, g^G)$ denotes the space of $2$-varifolds in $\R^3$, with
$$\int_{\R^3}\frac{1}{4\pi}e^{-\frac{|x|^2}{4}}d\|V\|<\infty.$$

\item Given $V\in \V_2^G(\R^3)$, $V^G$ denotes $\Gau V$. We also view $V^G$ as the extension of $\Gau V$ to $\V_2(S^3)$ by defining $\|V^G\|(\{\infty\})=0$. Thus, $\V_2^G(\R^3)\subset \V_2(S^3)$.

\item $\bF_{S^3}$ denotes the $\bF$-metric on $\V_2(S^3)$ \cite[\S 2.1(19)]{P81}, i.e. given $V, W\in \V_2(S^3)$,
\begin{displaymath}
\begin{split}
\bF_{S^3}(V, W)=\sup\{ & \int_{G_2(S^3)}f(x, S)dV(x, S)-\int_{G_2(S^3)}f(x, S)dW(x, S),\\
                                      & f\in C^0\big(G_2(S^3)\big), |f|\leq 1, Lip(f)\leq 1\},
\end{split}
\end{displaymath}
where $Lip(f)$ is the Lipschitz constant of $f$ with respect to the induced metric on $G_2(S^3)$ by $ds_0^2$.

\item $U_r(V)$ denotes the ball in $\V_2(S^3)$ with respect to $\bF_{S^3}$ metric with center $V\in\V_2(S^3)$ and radius $r>0$.

\end{itemize}


\subsection{First variation of 2-varifolds in $\V_2(\R^3, g^G)$}

Given a $C^1$ map $f: \R^3\rightarrow \R^3$, the {\em Jacobian} of $f$ with respect to the Gaussian metric $g^G=\Gau dx^2$ is given by
\begin{equation}\label{Jacobian in Gaussian metric}
J^G f(x, S)=Jf(x, S)\frac{e^{-|f(x)|^2/4}}{e^{-|x|^2/4}}, \quad (x, S)\in G_2(\R^3),
\end{equation}
where $Jf(x, S)$ is the Jacobian of $f$ with respect to the Euclidean metric $dx^2$.  Given $V\in\V_2(\R^3, g^G)$, the {\em push-forward} $f_{\#}(V^G)$ of $V^G=\Gau V$ under $f$ (c.f. \cite[\S 39]{Si83}) is given by:
\begin{equation}\label{push-forward in Gaussian metric}
f_{\#}(V^G)(A)=\int_{F^{-1}(A)}J^G f(x, S)\Gau dV(x, S),
\end{equation}
where $A\subset G_2(\R^3)$, and $F: G_2(\R^3)\rightarrow G_2(\R^3)$ is given by $F(x, S)=\big(f(x), df_x(S)\big)$, $(x, S)\in G_2(\R^3)$.

Given $X\in \X_c(\R^3)$, let $f_t: \R^3\rightarrow \R^3$ be the flow given by $X$, i.e. $\frac{d}{dt}f_t(p)=X(f_t(p))$ and $f_0=id$. Given $V\in \V_2(\R^3, g^G)$, the {\em first variation formula} of $V^G=\Gau V$ is
\begin{equation}\label{1st variation 1}
\begin{split}
\de_G V^G(X) & := \frac{d}{dt}\bigg{|}_{t=0}\|(f_t)_{\#}(V^G)\|(\R^3)\\
                   & =\frac{d}{dt}\bigg{|}_{t=0} \int_{G_2(\R^3)} J^G f_t(x, S) \Gau dV(x, S)\\
                   & =\frac{d}{dt}\bigg{|}_{t=0} \int_{G_2(\R^3)} Jf_t(x, S) \frac{1}{4\pi}e^{-\frac{|f_t(x)|^2}{4}}dV(x, S)\\
                   & =\int_{G_2(\R^3)}\big(div_S X-\frac{\lan X, \vec{x}\ran}{2}\big) dV^G(x, S).
\end{split}
\end{equation}
Here $\vec{x}$ is the position vector of $x$ in $\R^3$, and $div_S X$ is the divergence of the vector field $X$ on a given 2-plane $S$ with respect to the Euclidean metric $dx^2$, i.e. $div_S X=\sum_{i=1}^2\lan \nabla_{e_i}X, e_i\ran$, where $\{e_1, e_2\}$ is an orthonormal basis of $S$ under $dx^2$.

In fact, we can also get (\ref{1st variation 1}) by using the conformally changed metric $g^G=\Gau dx^2$. By basic first variation formula for sub-manifolds (c.f. \cite[Chap. 1, \$ 1.3]{CM11}),
\begin{equation}\label{1st variation 2}
\begin{split}
\de_G V^G(X) & =\frac{d}{dt}\bigg{|}_{t=0} \int_{G_2(\R^3)} J^G f_t(x, S) \Gau dV(x, S)\\
                        & =\int_{G_2(\R^3)}div^G_S X dV^G(x, S).
\end{split}
\end{equation}
Here $div^G_S X$ is the divergence of the vector field $X$ on a given 2-plane $S$ with respect to the Gaussian metric $\Gau dx^2$. It is easily seen that (c.f. \cite[\$ 1.159]{Be})
$$div^G_S X=div_S X-\frac{\lan X, \vec{x}\ran}{2}.$$

\begin{definition}\label{definition of F-stationary}
$V^G$ is called $F$-stationary, if $\de_G V^G(X)=0$ for any $X\in \X_c(\R^3)$.
\end{definition}



Given $V\in \V_2(S^3)$, the restriction of $V$ to $G_2(\R^3)$, i.e. $V\lc G_2(\R^3)$, is a 2-varifold in $\V_2(\R^3)$. We will use the first variation of $V\lc G_2(\R^3)$ under the Gaussian metric $g^G$. We abuse the notion of first variation, and write $\de_G\big(V\lc G_2(\R^3)\big)$ as $\de_G V$. By (\ref{1st variation 1})(\ref{1st variation 2}), 
\begin{equation}\label{1st variation 3}
\begin{split}
\de_G V(X) & =\de_G\big(V\lc G_2(\R^3)\big)(X)\\
                   & =\int_{G_2(\R^3)}\big(div_S X-\frac{\lan X, \vec{x}\ran}{2}\big) dV(x, S)\\
                   & =\int_{G_2(\R^3)}div^G_S X dV(x, S).
\end{split}
\end{equation}


\section{Overview of the proof of the min-max theorem}\label{overview of the proof of min-max theorem}

Given a saturated set $\La$ of n-parameter continuous families of surfaces, we will outline the proof of Theorem \ref{main min-max theorem} in this section. The proof consists of three parts.

\subsection{Pull-tight} 
In general, given a minimizing sequence $\{\{\Si_{\nu}\}^k\}$, viewing each slice $\Si^k_{\nu}$ as a varifold in $\V^G_2(\R^3)$ by multiplying with the Gaussian weight $\frac{1}{4\pi}e^{-\frac{|x|^2}{4}}$, it is easy to find a min-max sequence which converges to a $F$-stationary varifold $V$ under the varifold norm $\bF_{S^3}$ on $S^3$, and $\|V\|(S^3)=W(\La)$. However we do not want $\|V\|(\{\infty\})\neq 0$, and it is not necessarily true that every min-max sequence $\{\Si^k_{\nu_k}\}$ converges to an $F$-stationary varifold. We will deal with this difficulty by a carefully designed ''pull-tight" argument (the original argument on a compact manifold is due to F. Almgren \cite{AF65} and J. Pitts \cite{P81}). Our version uses the framework of Colding-De Lellis \cite[\S 4]{CD}. As the restriction of our families to the boundary $\partial I^n$ are nontrivial surfaces, we will actually use a multi-parameter version similar to the compact case in \cite[\S 15]{MN12}.
\begin{proposition}\label{pull tight}
If $W(\La)>\max_{\nu\in \partial I^n}F(\Si_{\nu})$, then there exists a minimizing sequence $\{\{\Si_{\nu}\}^k\}\subset\La$ such that every min-max sequence $\{\Si^k_{\nu_k}\}$ converges to a $F$-stationary varifold $V$ under the $\bF_{S^3}$-norm with $\|V\|(\R^3)=W(\La)$.
\end{proposition}
We call such a varifold $V$ a {\em min-max varifold}.
\begin{remark}
Compared to the arguments in \cite{AF65, P81, CD}, the main difficulty in our case is due to the fact that the underlying space $\R^3$ is non-compact. To overcome this issue, we view all the Gaussian weighted varifolds as varifolds defined on $S^3=\R^3\cup\{\infty\}$. Another difficulty is that the limit of a sequence of such varifolds might have a point mass at $\infty$. We will deal with this by a specially designed tightening process in the following sections \S \ref{An important vector field to deform the delta-mass} to \S \ref{Construction of tightening map}. The final proof will be given in \S \ref{Deforming smooth families by the tightening map}. 
\end{remark}

\subsection{Almost minimizing}
The regularity of the min-max varifold follows from the concept of "almost minimizing surfaces", or a.m. surfaces, developed by F. Almgren \cite{AF65}. We will use the version by Colding-De Lellis \cite{CD}. Denote $C$ by a fixed integer.

\begin{definition}
(\cite[Definition 3.2]{CD}) Given $\ep>0$, an open set $U\subset \R^3$, and a surface $\Si$, we say that $\Si$ is $\ep$-a.m. in $U$ if there {\em DOES NOT} exist any isotopy $\psi$ supported in $U$ such that
$$F(\psi(t, \Si))\leq F(\Si)+\ep/C,\textrm{for all } t;$$
$$F(\psi(1, \Si))\leq F(\Si)-\ep.$$
A sequence of surfaces $\{\Si^n\}$ is said to be a.m. in $U$ if each $\Si^n$ is $\ep_n$-a.m. in $U$ for some sequence $\ep_n\rightarrow 0$.
\end{definition}

Using the combinatorial arguments of F. Almgren \cite{AF65} and J. Pitts \cite{P81}, L. Simon and F. Smith (see Colding-De Lellis \cite{CD}) proved that one could always find at least one min-max sequence that is almost minimizing. Since the proof is essentially local, we can adapt them here in a straightforward way. The appendix of \cite{CGK} provides a multi-parameter version of these results, yielding:
\begin{proposition}\label{existence of a.m. sequence}
If $W(\La)>\max_{\nu\in \partial I^n}F(\Si_{\nu})$, then there exists a min-max sequence $\{\Si^j\}$ and a function $r: (\R^3, g^G)\rightarrow \R_+$ such that
\begin{itemize}
\vspace{-5pt}
\addtolength{\itemsep}{-0.7em}
\setlength{\itemindent}{0em}
\item $\{\Si^j\}$ is a.m. in every annulus $An$ (w.r.t. the Gaussian metric $g^G$) centered at $x$ and with outer radius at most $r(x)$;
\item $\frac{1}{4\pi}e^{-\frac{|x|^2}{4}}\Si^j$ converges to a $F$-stationary varifold $V$ in $S^3$ with $\|V\|(\{\infty\})=0$.
\end{itemize}
\end{proposition}

Now consider $\R^3$ with the conformally changed metric $g^G$. Our definition of a.m. sequence is then the same as in \cite{CD}. In fact, it was shown that the varifold limit of an a.m. sequence has smooth support. The proof is purely local, and is applicable to our case.
\begin{theorem}\label{regularity for a.m. surface}
\cite[Theorem 7.1]{CD} The support of $V$ is a smooth, embedded $F$-minimal surface (i.e. self-shrinker) $\Sigma$.  Thus $V=m\Sigma$ for some positive integer $m$.
\end{theorem}
\begin{remark}\label{remark for regularity for a.m. surface}
As the smooth metric measure space $(\R^3, dx^2, \frac{1}{4\pi}e^{-\frac{|x|^2}{4}}dvol)$ has positive Bakry-Emery Ricci tensor, we know that any two Gaussian minimal surfaces must intersect by a Frankel-type theorem \cite[Theorem 7.4]{WW}. So the support of the min-max varifold in Theorem \ref{regularity for a.m. surface} must be connected.
\end{remark}

\subsection{Genus control}

The min-max surface $\Sigma$ may {\it a priori} have infinite topology, so it is most convenient to consider exhaustions of the Gaussian space.  Fix a sequence of positive numbers $R_i\rightarrow\infty$ so that $\Sigma$ is transverse to $\partial B_{R_i}(0)$ and thus intersects $\partial B_{R_i}(0)$ in a union of $e_i$ circles and has genus $g_i$ in $B_{R_i}(0)$.  For each $i>0$, we can then find $2g_i$ curves $\{\gamma_j\}_{j=1}^{2g_i}$ on $\Sigma\cap B_{R_i}(0)$ meeting at one point so that $\Sigma\setminus\cup_{j=1}^{2g_i}\gamma_j$ is a planar domain with $e_i$ ends.  The lifting argument from \cite{K13} implies that $g_im\leq g$, (where $m$ is the integer multiplicity such that $V=m\Sigma$).   Since this holds for all positive $i$, we see that the genus  $h$ of $\Sigma$ is finite and moreoever that $h$ is at most $g$ (in fact $hm\leq g$).


\section{An important vector field to deform the delta-mass}\label{An important vector field to deform the delta-mass}

One important issue to carry out the ``tightening" is to deform the delta-mass at $\{\infty\}$ in a continuous way. We will use the flow of a certain vector field to do the deformation. The vector field we will use is
$$X(x)=\frac{\vec{x}}{r^2},$$
where $r=|x|$ is the radial distance function of $(\R^3, dx^2)$. Now we collect a few properties of this vector field.

\subsection{A basic calculation of radial vector field}

First, we calculate:
\begin{equation}
\begin{split}
div^G_S X & =div_S\frac{\vec{x}}{r^2}-\frac{\lan\vec{x}/r^2, \vec{x}\ran}{2} =\frac{2}{r^2}-\frac{2}{r^3}\lan \vec{x}, \nabla_S r\ran-\frac{1}{2}\\
                  & =\frac{2}{r^2}-\frac{2}{r^2}|\nabla_S r|^2-\frac{1}{2}\\
                  & =\frac{2|\nabla^{\perp}r|^2}{r^2}-\frac{1}{2}.
\end{split}
\end{equation}
Here $\nabla_S r$ is the projection of the gradient $\nabla r$ onto the 2-plane S, and $\nabla^{\perp}r$ is the projection of $\nabla r$ to the orthogonal complement of $S$ (with respect to the Euclidean metric $dx^2$).

The calculation directly implies that,
\begin{lemma}
$div^G_S X(x)$ is a bounded function near $\infty$ in $G_2(\R^3)$, and can be extended to a $C^0$ function in $G_2(S^3)$ away from $0$ by letting it equal to $-\frac{1}{2}$ at $\infty$.
\end{lemma}

\subsection{Radial vector field with compact support}

We usually need to multiply $X$ with a cutoff function to make it supported near $\infty$. Given a number $\rho>0$, let $\ga(r)=\phi(\frac{r}{\rho})$, where $\phi: \R\rightarrow \R_+$ is a smooth function such that
\begin{equation}
\left. \phi(s)\Bigg\{ \begin{array}{ll}
=1, \quad \textrm{ if $s\geq 2$}\\
=0, \quad \textrm{ if $s\leq 1$}\\
\geq 0, \quad \textrm{ if $1\leq s\leq 2$}
\end{array}\right. \textrm{ and } 0\leq \phi^{\pr}(s) \leq 1+\ep,
\end{equation}
for some small $\ep>0$.

Let
\begin{equation}\label{radial vector field}
X(x)=\phi(\frac{r}{\rho})\frac{\vec{x}}{r^2}.
\end{equation}
Then the gradient of $X$ is given by
$$DX=\phi^{\prime}(\frac{r}{\rho})\frac{Dr}{\rho}\otimes \frac{x}{r^2}+\phi(\frac{r}{\rho})\big(\frac{Dx}{r^2}-\frac{2Dr\otimes x}{r^3}\big).$$
Using the definition of $\phi$, it is easily seen that
$$\|DX\|_{L^{\infty}}\leq (4+\ep)\frac{1}{\rho^2}.$$
\begin{lemma}\label{C1 bound for X}
For $\rho>3$, $\|DX\|_{L^{\infty}}\leq 1$.
\end{lemma}
Moreover
\begin{equation}\label{divergence of radial vector field}
\begin{split}
div^G_S X & =\phi(\frac{r}{\rho})div^G_S(\frac{\vec{x}}{r^2})+\lan \nabla_S\phi(\frac{r}{\rho}), \frac{\vec{x}}{r^2}\ran\\
                  & =\phi(\frac{r}{\rho})\big(\frac{2|\nabla^{\perp}r|^2}{r^2}-\frac{1}{2}\big)+\frac{1}{\rho}\phi^{\pr}(\frac{r}{\rho})\frac{\lan \nabla_S r, \vec{x}\ran}{r^2}\\
                  & =\phi(\frac{r}{\rho})\big(\frac{2|\nabla^{\perp}r|^2}{r^2}-\frac{1}{2}\big)+\frac{r}{\rho}\phi^{\pr}(\frac{r}{\rho})\frac{|\nabla_S r|^2}{r^2}.
\end{split}
\end{equation}
By our choice of $\phi$, we know that
$$0\leq \frac{r}{\rho}\phi^{\pr}(\frac{r}{\rho})\leq \frac{r}{\rho}(1+\ep)\leq 2(1+\ep), \textrm{as $\phi^{\pr}\neq 0$ only when $1\leq \frac{r}{\rho}\leq 2$}.$$
Therefore the asymptotic behavior of $div^G_S X$ near $\infty$ is as follows:
\begin{equation}\label{asymptotics for divergence of radial vf}
\left. div^G_S X\ \Bigg\{ \begin{array}{ll}
=\frac{2|\nabla^{\perp}r|^2}{r^2}-\frac{1}{2}, \quad \textrm{ if $r\geq 2\rho$}\\
=0, \quad\quad\quad\quad\quad \textrm{ if $r\leq \rho$}\\
\leq  \phi(\frac{r}{\rho})\big(\frac{2|\nabla^{\perp}r|^2}{r^2}-\frac{1}{2}\big)+2(1+\ep)\frac{|\nabla_S r|^2}{r^2}, \quad \textrm{ if $\rho\leq r\leq 2\rho$}.
\end{array}\right.
\end{equation}

\subsection{Lipschitz bound for $div^G_S X$}

In this section, we will show that $div^G_S X$ extends to a Lipschitz function on $G_2(S^3)$ with respect to the round metric $ds^2_0$. Here we use the {\em stereographic projection}
\begin{equation}\label{stereographic projection}
(x, y, z)\rightarrow \big(\frac{2x}{1+r^2}, \frac{2y}{1+r^2}, \frac{2z}{1+r^2}, \frac{r^2-1}{r^2+1}\big)
\end{equation}
to identify $\R^3$ with $S^3\backslash\{\textrm{north pole}\}$. Under this map,
$$ds^2_0=\frac{4}{(1+r^2)^2}dx^2.$$
\begin{lemma}
lf $f$ is a function defined on $\R^3$, then 
\begin{equation}\label{conformal change of gradient}
|\nabla_0 f|^2_{ds^2_0}=\frac{(1+r^2)^2}{4}|\nabla f|^2_{dx^2},
\end{equation}
where $\nabla_0$ is the connection of $ds^2_0$, and $f$ is viewed as a function on $S^3$ on the left hand side.
\end{lemma}

Denote $f(x, S)=div^G_S X$ (\ref{divergence of radial vector field}).
\begin{lemma}\label{Lipschitz bound}
$f(x, S)$ extends to a Lipschitz function on $G_2(S^3)$ with respect to the round metric $ds^2_0$.
\end{lemma}

\begin{remark}
This fact shows that the first variation $\de_G V(X)$ (see \ref{1st variation 4}) depends continuously on $V$ under the $\bF_{S^3}$ norm.
\end{remark}

\begin{proof}
It is easy to see that the Lipschitz constant of $f(x, S)$ with respect to the $S$-variable is uniformly bounded on $\R^3$. As the fiber part of $G_2(\R^3)$ and $G_2(S^3)$ are isometric, to show that $f(x, S)$ extends to a Lipschitz function on $G_2(S^3)$, we only need to show that $|\nabla_0 f(x, S_0)|^2_{ds^2_0}$ is uniformly bounded for a fixed 2-plane $S_0$. Here we use (\ref{conformal change of gradient}) to connect $|\nabla_0 f(x, S_0)|^2_{ds^2_0}$ to $|\nabla f(x, S_0)|^2_{dx^2}$. In fact,
\begin{displaymath}
\begin{split}
\nabla f &= \phi(\frac{r}{\rho})\big(2\nabla \frac{|\nabla^{\perp} r|^2}{r^2}\big)+\phi^{\pr}(\frac{r}{\rho})\big(\frac{2|\nabla^{\perp} r|^2}{r^2}-\frac{1}{2}\big)\frac{\nabla r}{\rho}+\phi^{\pr}(\frac{r}{\rho})\frac{|\nabla_S r|^2}{r^2}\frac{\nabla r}{\rho}\\
             &+\frac{r}{\rho}\phi^{\pr\pr}(\frac{r}{\rho})\frac{|\nabla_S r|^2}{r^2}\frac{\nabla r}{\rho}+\frac{r}{\rho}\phi^{\pr}(\frac{r}{\rho})\big(\nabla \frac{|\nabla_S r|^2}{r^2}\big).
\end{split}
\end{displaymath}
In the above formula, terms with compact support in $\R^3$, i.e. those containing $\phi^{\pr}$ or $\phi^{\pr\pr}$, are all bounded with respect to the round metric $ds^2_0$. Therefore, we only need to take care of the first term, where
$$\nabla \frac{|\nabla^{\perp} r|^2}{r^2}=-2|\nabla^{\perp} r|^2\frac{\nabla r}{r^3}+\frac{2}{r^2}\lan \nabla^{\perp} r, \nabla\nabla^{\perp} r\ran.$$
Here $\nabla\nabla r$ is asymptotic to $\frac{1}{r}$ when $r\rightarrow\infty$, so
$$\big|\nabla \frac{|\nabla^{\perp} r|^2}{r^2}\big|^2_{dx^2}\textrm{ is asymptotic to }\frac{1}{r^6} \textrm{ as } r\rightarrow \infty.$$
Using (\ref{conformal change of gradient}), $|\nabla _0 f|^2_{ds^2_0}$ is asymptotic to $\frac{(1+r^2)^2}{r^6}$ as $r\rightarrow \infty$. So $|\nabla _0 f|^2_{ds^2_0}$ is uniformly bounded on $\R^3$, and the proof is finished.
\end{proof}

\subsection{Push-forward of 2-varifolds in $\V_2^G(\R^3)$ by radial vector field}

In this part, we give an explicit expression for the push-forward of 2-varifolds in $\V_2(\R^3)$ under the flow of the radial vector field $X=\frac{\vec{x}}{r^2}$. Let
\begin{equation}\label{flow of radial vf}
f_t: x\rightarrow \sqrt{1+\frac{2t}{r^2}}x
\end{equation}
be the flow associated with $X$.
\begin{lemma}
The Jacobian of $f_t$ under the Gaussian metric $g^G$ is given by
\begin{equation}\label{Jacobian of radial flow 1}
J^G f_t(x, S)=\Big(1+\big[(1+\frac{2t}{r^2})^2-1\big]|\nabla^{\perp}r|^2\Big)^{1/2}e^{-\frac{t}{2}},
\end{equation}
where $\nabla^{\perp}r$ is the projection of $\nabla r$ onto the orthogonal complement of $S$ under the Euclidean metric $dx^2$. Therefore, given $V\in \V_2^G(\R^3)$, then for any $A\subset\subset G_2(\R^3\backslash\{0\})$,
\begin{equation}\label{push forward of radial flow 1}
(f_t)_{\#}(V^G)(A)=\int_{F_t^{-1}(A)}\Big(1+\big[(1+\frac{2t}{r^2})^2-1\big]|\nabla^{\perp}r|^2\Big)^{1/2}e^{-\frac{t}{2}}dV^G(x, S),
\end{equation}
where $F_t(x, S)=\big(f_t(x), d(f_t)_xS\big)$.
\end{lemma}
\begin{remark}\label{remark of push forward 1}
In the following, we will use the flows $\ti{f}_t$ defined by vector fields which are equal to $\frac{\vec{x}}{r^2}$ near $\infty$, e.g. (\ref{radial vector field}). The Jacobian $J^G\ti{f}_t$ and the push-forward formula are the same as those of $f_t$ around $\infty$.
\end{remark}
\begin{proof}
First, the Jacobian $Jf_t(x, S)$ of $f_t$ with respect to the Euclidean metric $dx^2$ is given by (see \S \ref{calculation of Jacobian of ft}):
$$Jf_t(x, S)=\Big(1+\big[(1+\frac{2t}{r^2})^2-1\big]|\nabla^{\perp}r|^2\Big)^{1/2}.$$
Then by (\ref{Jacobian in Gaussian metric}),
$$J^G f_t(x, S)=Jf_t(x, S)\frac{e^{-\big|\sqrt{1+\frac{2t}{r^2}}x\big|^2/4}}{e^{-|x|^2/4}}=Jf_t(x, S)e^{-\frac{t}{2}}.$$
The lemma then follows from (\ref{push-forward in Gaussian metric}).
\end{proof}

\subsection{Push-forward of 2-varifolds in $\V_2(S^3)$ by radial vector field}

Using the push-forward formula (\ref{push forward of radial flow 1}), we can define the ``push-forward" of 2-varifolds in $\V_2(S^3)$ by $f_t$ as follows. In fact, the integrand in (\ref{push forward of radial flow 1}) has a limit as $r \rightarrow\infty$, i.e.
$$\lim_{r\rightarrow\infty}\Big(1+\big[(1+\frac{2t}{r^2})^2-1\big]|\nabla^{\perp}r|^2\Big)^{1/2}e^{-\frac{t}{2}}=e^{-\frac{t}{2}}.$$
So the integrand extends to a continuous function around $\infty$ in $G_2(S^3)$ by identifying $S^3\cong \R^3\cup\{\infty\}$. Also, we can extend the tangential map $F_t: G_2(\R^3)\rightarrow G_2(\R^3)$ by letting $F_t(\infty, S)=(\infty, S)$, and $F_t$ is continuous from $G_2(S^3)$ to $G_2(S^3)$ by (\ref{differential of ft}). Therefore, we have
\begin{definition}\label{push forward of 2-varifolds on S3 by radial vector field}
Given $V\in \V_2(S^3)$, the {\em ``push-forward"} of $V$ under $f_t$ (\ref{flow of radial vf}) is defined as:
\begin{equation}\label{push forward of radial flow 2}
(f_t)_{\#}V(A)=\int_{F_t^{-1}(A)}\Big(1+\big[(1+\frac{2t}{r^2})^2-1\big]|\nabla^{\perp}r|^2\Big)^{1/2}e^{-\frac{t}{2}}dV(x, S),
\end{equation}
for any $A\subset\subset G_2\big((\R^3\cup\{\infty\})\backslash\{0\}\big)$.
\end{definition}

As in Remark \ref{remark of push forward 1}, we also need to define the push-forward of varifolds in $\V_2(S^3)$ by flows $\ti{f}_t$ defined by vector fields which are equal to $\frac{\vec{x}}{r^2}$ near $\infty$. As $\ti{f}_t=f_t$ near $\infty$, the Jacobian $J^G\ti{f}_t(x, S)$ is equal to the Jacobian $J^G f_t(x, S)$ around $\infty$, so $J^G\ti{f}_t(x, S)$ also extends to a continuous function on $G_2(S^3)=G_2(\R^3\cup\{\infty\})$. Also the tangent map $\ti{F}_t$ extends to a continuous map from $G_2(S^3)\rightarrow G_2(S^3)$ as $F_t$. We can define the push-forward $(\ti{f}_t)_{\#}V$ in the same way, i.e.
\begin{definition}\label{push forward of 2-varifolds on S3}
For any $V\in \V_2(S^3)$ and $A\subset G_2(S^3)$,
\begin{equation}\label{push forward of radial flow 3}
(\ti{f}_t)_{\#}V(A)=\int_{\ti{F}_t^{-1}(A)}J^G\ti{f}_t(x, S)dV(x, S).
\end{equation}
\end{definition}

Also we have the following corollary,
\begin{lemma}\label{continuity of push forward by radial flow}
The map $t\rightarrow (\ti{f}_t)_{\#}V$ is continuous from $\R^+$ to $\V_2(S^3)$.
\end{lemma}
\begin{proof}
We only need to show that the maps $t\rightarrow J^G\ti{f}_t(x, S)$ and $t\rightarrow \ti{F}_t$ are continuous from $\R^+$ to $C^0\big(G_2(S^3)\big)$ and $C^0\big(G_2(S^3), G_2(S^3)\big)$ under the $L^{\infty}$-norms respectively.

Given any compact subset $K$ of $\R^3$, the map $t\rightarrow J^G\ti{f}_t(x, S)$ is continuous under the $C^k\big(G_2(K)\big)$-norm ($k\geq 0$) by usual ODE theory. Given a compact neighborhood $K^{\pr}$ of $\infty$ (in $S^3$) where $\ti{f}_t=f_t$, by (\ref{Jacobian of radial flow 1}), the map $t\rightarrow J^G\ti{f}_t(x, S)$ is continuous under $L^{\infty}(G_2(K^{\pr}))$-norm. The continuity of $t\rightarrow J^G\ti{f}_t(x, S)$ follows by combining the continuity on $K$ and $K^{\pr}$.

Note that $\ti{F}_t(x, S)=(\ti{f}_t(x), d(\ti{f}_t)_x S)$. The continuity of the map $t\rightarrow \ti{f}_t(x)$ from $\R^+$ to $C^0(S^3, S^3)$ follows from the same argument as above. The continuity of the map $t\rightarrow d(\ti{f}_t)_x(S)$ from $\R^+$ to $C^0\big(G_2(S^3), G_2(S^3)\big)$ also follows from similar argument as above but using (\ref{differential of ft}) in place of (\ref{Jacobian of radial flow 1}). 
\end{proof}

We have defined the first variation for 2-varifolds in $\V_2(S^3)$ with respect to $g^G$ by restricting to $\R^3=S^3\backslash\{\infty\}$ (\ref{1st variation 3}). For the special radial vector fields  $X$ (\ref{radial vector field}), using this notion of ``push-forward" in (\ref{push forward of radial flow 3}), we can define the {\em first variation formula} of 2-varifold $V\in \V_2(S^3)$ with respect to $g^G$ on $S^3$.
\begin{equation}\label{1st variation 4}
\begin{split}
\de_G V(X) &:= \frac{d}{dt}\bigg{|}_{t=0}\|(\ti{f}_t)_{\#}V\|(\R^3\cup \{\infty\})\\
                   & = \frac{d}{dt}\bigg{|}_{t=0}\int_{G_2(\R^3\cup\{\infty\})}J^G\ti{f}_t(x, S)dV(x, S)\\
                   & = \int_{G_2(\R^3\cup\{\infty\})}div^G_S X(x) dV(x, S)\\
                   & =\int_{G_2(\R^3\cup\{\infty\})}\Big[\phi(\frac{r}{\rho})\big(\frac{2|\nabla^{\perp}r|^2}{r^2}-\frac{1}{2}\big)+\frac{r}{\rho}\phi^{\pr}(\frac{r}{\rho})\frac{|\nabla_S r|^2}{r^2}\Big]dV(x, S).
 \end{split}
\end{equation}
Here $\lim_{r\rightarrow\infty}div^G_S X(x)=-\frac{1}{2}$ by (\ref{asymptotics for divergence of radial vf}).


\section{Constructing tightening vector field}\label{Constructing tightening vector field}

Given a continuous family $\{\Si_{\nu}\}_{\nu\in I^n}$ and the associated saturated set $\La$ with min-max value $W(\La)$, let $L=2W(\La)>0$. Consider the set of 2-varifolds in $\V_2(S^3)$ with bounded mass: $A=\{V\in \V_2(S^3):\ \|V\|(S^3)\leq L\}$. Let $B=\{\Si^G_{\nu}: \nu\in\partial I^n\}$ ($\Si^G_{\nu}$ denotes the weighted varifold, c.f. \S \ref{notation}). Denote
$$\bar{A}_0=\big\{V\in A: \de_G V=0 \textrm{ on } \R^3\cong S^3\backslash \{\infty\}, \textrm{ and }\|V\|(\{\infty\})=0\big\}.$$
Let $A_0=\bar{A}_0\cup B$. Consider the concentric annuli around $A_0$ under the $\bF_{S^3}$-metric, i.e.
$$A_1=\big\{V\in A: \bF_{S^3}(V, A_0)\geq 1\big\}$$
$$A_j=\big\{V\in A: \frac{1}{2^j}\leq \bF_{S^3}(V, A_0)\leq \frac{1}{2^{j-1}}\big\}, \quad j\in\N, j\geq 2.$$

\begin{lemma}\label{A0 is compact}
$A_0$ is a compact subset of $A$ under the $\bF_{S^3}$-metric.
\end{lemma}
\begin{proof}
We only need to show that $\bar{A}_0$ is compact. To prove this, we only need to show that every $V\in \bar{A}_0$ has uniformly small mass near $\infty\in S^3$. Denote $V=\frac{1}{4\pi}e^{-\frac{|x|^2}{4}}V^{\pr}$, where $V^{\pr}\in \V_2^G(\R^3)$; then $V^{\pr}$ is $F$-stationary by the definition of $\bar{A}_0$.

Letting $(x_0, t_0)=(0, R^2)$ for $R\gg 1$ in (\ref{Gaussian volume at (x t)}), and using Theorem \ref{entropy achieve at (0 1)}, we have
$$F_{(0, R^2)}(V^{\pr})\leq \la(V^{\pr})=F_{(0, 1)}(V^{\pr})=\int_{\R^3}\frac{1}{4\pi}e^{-\frac{|x|^2}{4}}d\|V^{\pr}\|=\|V\|(S^3).$$
On the other hand,
\begin{displaymath}
\begin{split}
F_{(0, R^2)}(V^{\pr}) & = \int_{\R^3}\frac{1}{4\pi R^2}e^{-\frac{|x|^2}{4R^2}}d\|V^{\pr}\| = \int_{\R^3}\frac{1}{4\pi R^2}e^{-\frac{|x|^2}{4R^2}}\cdot 4\pi e^{\frac{|x|^2}{4}}d\|V\|\\
                                  & \geq \int_{|x|\geq R}\frac{1}{R^2}e^{(1-\frac{1}{R^2})\frac{|x|^2}{4}}d\|V\|\\
                                  & \geq \frac{1}{R^2}e^{(1-\frac{1}{R^2})\frac{R^2}{4}}\|V\|(B_R^c),
\end{split}
\end{displaymath}
where $B_R^c$ is the complement of the ball $B_R$ centered at $0$ of radius $R$.

Combining the above, we have
$$\|V\|(B_R^c)\leq \|V\|(S^3)R^2e^{-\frac{R^2}{4}(1-\frac{1}{R^2})},$$
which shows that $V$ has uniformly small mass near $\infty$, and hence finish the proof.
\end{proof}

In the next lemma, we show that for any 2-varifold in $A_j$, we can find a vector field, along which the first variation is bounded from above by a fixed negative number depending only on $j$.
\begin{lemma}\label{vf to deform mass}
For any $V\in A_j$, there exists $X_V\in\X(\R^3)$, such that either $X_V\in\X_c(\R^3)$, or $X_V=\frac{\vec{x}}{r^2}$ near $\infty$, and
\begin{equation}\label{1st variation upper bounded}
\|X_V\|_{C^1(\R^3)}\leq 1, \quad \de_G V(X_V)\leq-c_j<0,
\end{equation}
for some $c_j$ depending only on $j$, and also $div^G_S X_V\in Lip\big(G_2(S^3)\big)$.
\end{lemma}
\begin{proof}
We separate the discussion into two cases:
\begin{itemize}
\vspace{-5pt}
\setlength{\itemindent}{3em}
\addtolength{\itemsep}{-0.7em}
\item[{\bf Case 1:}] $\|V\|(\{\infty\})<\frac{1}{2\cdot 2^j}$;
\item[{\bf Case 2:}] $\|V\|(\{\infty\})\geq \frac{1}{2\cdot 2^j}$.
\end{itemize}
\noindent{\bf Part \Rom{1}:} If $\|V\|(\{\infty\})<\frac{1}{2\cdot 2^j}$, we claim that $\bF_{S^3}(V\lc G_2(S^3\backslash\{\infty\}), A_0)>\frac{1}{2\cdot 2^j}$.

Let us first check the claim. Given $f\in Lip\big(G_2(S^3)\big)$, $|f|\leq 1$, $Lip(f)\leq 1$, then for any $W\in A_0$, $\|W\|(\{\infty\})=0$, and we have
\begin{displaymath}
\begin{split}
\bigg| \int_{G_2(S^3)} & f(x, S)dV(x, S) -\int_{G_2(S^3)}f(x, S)dW(x, S) \bigg| \\
                                    & \leq \bigg| \int_{G_2(S^3\backslash\{\infty\})}f(x, S)dV-\int_{G_2(S^3\backslash\{\infty\})}f(x, S)dW \bigg| +\bigg| \int_{G_2(\{\infty\})}f(\infty, S)dV\bigg|\\
                                    & < \bigg| \int_{G_2(S^3\backslash\{\infty\})}f(x, S)dV-\int_{G_2(S^3\backslash\{\infty\})}f(x, S)dW \bigg| +\frac{1}{2\cdot 2^j}.
\end{split}
\end{displaymath}
By taking a supreme of the above inequality over all $f\in Lip\big(G_2(S^3)\big)$, $|f|\leq 1$, $Lip(f)\leq 1$,
$$\frac{1}{2^j}\leq \bF_{S^3}(V, W)< \bF_{S^3}(V\lc G_2(\R^3), W)+\frac{1}{2\cdot 2^j}.$$
The claim then follows from the above inequality.

By the claim, $V\lc G_2(\R^3)$ is not $F$-stationary, so there exists $X_V\in\X_c(\R^3)$, $\|X_V\|_{C^1(\R^3)}\leq 1$, such that
$$\de_GV(X_V)<0.$$
Moreover, $div^G_SX_V$ is a Lipschitz function on $G_2(S^3)=G_2(\R^3\cup\{\infty\})$ as $X_V$ has compact support in $\R^3$.

\vspace{5pt}
\noindent{\bf Part \Rom{2}:} If $\|V\|(\{\infty\})\geq \frac{1}{2\cdot 2^j}$, then we can find $\rho>0$ large enough, such that $\|V\|\big(B_{2\rho}\backslash B_{\rho}\big)\leq \frac{1}{100\cdot 2\cdot 2^j}$, where $B_{\rho}$ is the ball of $\R^3$ centered at $0$ with radius $\rho$ under the Euclidean metric $dx^2$. Consider $X_V(x)=\phi(\frac{r}{\rho})\frac{\vec{x}}{r^2}$ defined in (\ref{radial vector field}) for the chosen $\rho$. Then by (\ref{1st variation 4}) and (\ref{asymptotics for divergence of radial vf}),
\begin{displaymath}
\begin{split}
\de_G V(X_V) & = \int_{G_2(\R^3\cup\{\infty\})}\Big[\phi(\frac{r}{\rho})\big(\frac{2|\nabla^{\perp}r|^2}{r^2}-\frac{1}{2}\big)+\frac{r}{\rho}\phi^{\pr}(\frac{r}{\rho})\frac{|\nabla_S r|^2}{r^2}\Big]dV(x, S)\\
                   & \leq \int_{G_2(\R^3\backslash B_{2\rho} \cup\{\infty\})}\big(\frac{2|\nabla^{\perp} r|^2}{r^2}-\frac{1}{2}\big) dV(x, S)\\
                   & +\int_{G_2(B_{2\rho}\backslash B_{\rho})} \Big[\phi(\frac{r}{\rho})\big(\frac{2|\nabla^{\perp}r|^2}{r^2}-\frac{1}{2}\big)+2(1+\ep)\frac{|\nabla_S r|^2}{r^2}\Big] dV(x, S)\\
                   & \leq -\frac{1}{4}\|V\|(\{\infty\})+5\|V\|(B_{2\rho}\backslash B_{\rho})\\
                   & <-\frac{1}{8}\frac{1}{2\cdot 2^j}.
\end{split}
\end{displaymath}
By Lemma \ref{C1 bound for X}, we can choose $\rho\geq 3$ such that $\|X_V\|_{C^1(\R^3)}\leq 1$. By Lemma \ref{Lipschitz bound}, $div^G_S X_V$ is Lipschitz on $G_2(S^3)$.

The upper bound (\ref{1st variation upper bounded}) follows from the compactness of $A_j$ under the $\bF_{S^3}$-norm and the continuity of the map $W\rightarrow \de_G W(X_V)=\int_{G_2(S^3)}div^G_S X_V dW(x, S)$ (with respect to the $\bF_{S^3}$-metric) for fixed $X_V$.
\end{proof}


\section{Construction of tightening map}\label{Construction of tightening map}

Using the preliminary results in the above, we can construct the tightening map similar to that in \cite[\S 4]{CD}\cite[\S 4.3]{P81} and \cite[\S 15]{MN12}.


\subsection{A map from $A$ to the space of vector fields}

In this section, we will construct a map $H: A\rightarrow \X(\R^3)$, which is continuous with respect to the 
$C^1$ topology on $\X(\R^3)$.

Given $V\in A_j$, let $X_V$ be given in Lemma \ref{vf to deform mass}. In both cases in Lemma \ref{vf to deform mass}, $div^G_S X_V$ is Lipschitz on $G_2(S^3)$. So for fixed $X_V$, the map $W\rightarrow \de_G W(X_V)=\int_{G_2(S^3)}div^G_S X_V dW(x, S)$ is continuous with respect to the $\bF_{S^3}$-metric. Therefore for any $V\in A_j$, there exists $0<r_V<\frac{1}{2^{j+1}}$, such that for any $W\in U_{r_V}(V)$ (c.f. \S \ref{notation}), i.e. $\bF_{S^3}(W, V)<r_V$,
\begin{equation}\label{1st variation upper bound 2}
\de_G W(X_V)\leq \frac{1}{2}\de_G V(X_V)\leq -\frac{1}{2}c_j<0.
\end{equation}
Now $\big\{U_{r_V/2}(V): V\in A_j\big\}$ is an open covering of $A_j$. By the compactness of $A_j$, we can find finitely many balls $\big\{U_{r_{j, i}}(V_{j, i}): V_{j, i}\in A_j, 1\leq i\leq q_j\big\}$,  such that
\begin{itemize}
\vspace{-5pt}
\setlength{\itemindent}{1em}
\addtolength{\itemsep}{-0.7em}

\item[$(\rom{1})$] The balls $U_{r_{j, i}/2}(V_{j, i})$ with half radii cover $A_j$;

\item[$(\rom{2})$] The balls $U_{r_{j, i}}(V_{j, i})$ are disjoint from $A_k$ if $|j-k|\geq 2$.
\end{itemize}
In the following, we denote $U_{r_{j, i}}(V_{j, i})$, $U_{r_{j, i}/2}(V_{j, i})$, $r_{V_{j, i}}$ and $X_{V_{j, i}}$ by $U_{j, i}$, $\ti{U}_{j, i}$, $r_{j, i}$ and $X_{j, i}$ respectively.

Now we can construct a partition of unity $\{\varphi_{j, i}: j\in\N, 1\leq i\leq q_j\}$ sub-coordinate to the covering $\big\{U_{r_{j, i}/2}(V_{j, i}): V_{j, i}\in A_j, 1\leq i\leq q_j, j\in\N\big\}$ by
$$\varphi_{j, i}(V)=\frac{\psi_{j, i}(V)}{\sum\{\psi_{p, q}(V), p\in\N, 1\leq q\leq q_p\}},$$
where $\psi_{j, i}(V)=\bF_{S^3}(V, A\backslash \ti{U}_{j, i})$.

The map $H: A\rightarrow \X(\R^3)$ is defined by
\begin{equation}\label{construction of vector fields H}
H_V=\bF_{S^3}(V, A_0) \sum_{j\in\N, 1\leq i\leq q_j}\varphi_{j, i}(V)X_{j, i}.
\end{equation}
The following lemma is a straightforward consequence of the construction.
\begin{lemma}\label{properties of HV}
The map $H: V\rightarrow H_V$ satisfies the following properties:
\begin{itemize}
\vspace{-5pt}
\setlength{\itemindent}{1em}
\addtolength{\itemsep}{-0.7em}

\item[$(0)$] $H_V=0$ if $V\in A_0$;

\item[$(\rom{1})$] For any $V\in A$, there exists $R>0$ large enough, such that $H_V=c_V\frac{\vec{x}}{r^2}$ outside $B(0, R)$ for some $0\leq c_V\leq 1$, and the map $V\rightarrow c_V$ is continuous;

\item[$(\rom{2})$] The map $H$ is continuous with respect to the 
$C^1$ topology on $\X(\R^3)$;

\end{itemize}
\end{lemma}


\subsection{A map from $A$ to the space of isotopies}\label{A map from A to the space of isotopies}

In this part, we will associate each $V\in A$ with an isotopy of $\R^3$ in a continuous manner. The isotopy will be generated by the vector field $H_V$. 

Given $V\in A$, we can construct a 1-parameter family of diffeomorphisms $\Phi_V: \R^+\times \R^3\rightarrow \R^3$ by
\begin{equation}\label{flow by vector fields}
\frac{\partial \Phi_V(t, x)}{\partial t}=H_V\big(\Phi_V(t, x)\big),\quad \Phi_V(0, x)=x.
\end{equation}
The solvability of the above ODE systems comes from Lemma \ref{properties of HV}(\rom{1}). In fact, the ODE systems is solvable on any compact subset of $\R^3$ by usual ODE theory. Near $\infty$, the solution is given by (\ref{flow of radial vf})\footnote{For $X(x)=c\frac{\vec{x}}{r^2}$, $f_t(x)=\sqrt{1+\frac{2ct}{r^2}}x$.}.

We will transform $V$ by $\Phi_V(t)$ to get a 1-parameter family of varifolds $V(t)=\big(\Phi_V(t)\big)_{\#} V$ (this is well-defined by similar argument as in Definition \ref{push forward of 2-varifolds on S3}), and we will show that the mass of $V(t)$ can be deformed down by a fixed amount depending only on $\bF_{S^3}(V, A_0)$.

In fact, given $V\in A_j$, let $r(V)$ be the smallest radii of the balls $\ti{U}_{k, i}$ which contain $V$. As there are only finitely many balls $\ti{U}_{k, i}$ which intersect with $A_j$, we know that $r(V)\geq r_j>0$, where $r_j$ depends only on $j$. Then
$$U_{r(V)}(V)\subset\cap_{V\in \ti{U}_{k, i}}U_{k, i}.$$
As the sub-index $k$ of these $\ti{U}_{k, i}$ can only be $j-1, j, j+1$ by our choice of $\ti{U}_{p, q}$'s, using (\ref{1st variation upper bound 2}) and (\ref{construction of vector fields H}), we have that for any $W\in U_{r(V)}(V)$,
$$\de_G W(H_V)\leq -\frac{1}{2^{j+1}}\min\{c_{j-1}, c_j, c_{j+1}\}.$$
Then there are two continuous functions $g:\R^+\rightarrow \R^+$ and $r: \R^+\rightarrow \R^+$, such that
\begin{equation}\label{1st variation upper bound-continuous version}
\de_G W(H_V)\leq -g\big(\bF_{S^3}(V, A_0)\big),\quad \textrm{ if }\quad \bF_{S^3}(W, V)\leq r\big(\bF_{S^3}(V, A_0)\big).
\end{equation}
Next, we will construct a continuous time function $T: [0, \infty) \rightarrow [0, \infty)$, such that
\begin{itemize}
\vspace{-5pt}
\setlength{\itemindent}{1em}
\addtolength{\itemsep}{-0.7em}
\item[$(\rom{1})$] $\lim_{t\rightarrow 0}T(t)=0$, and $T(t)>0$ if $t\neq 0$;

\item[$(\rom{2})$] For any $V\in A$, denote $\ga=\bF_{S^3}(V, A_0)$; then $V(t)=\big(\Phi_V(t)\big)_{\#}V\in U_{r(\ga)}(V)$ for all $0\leq t\leq T(\ga)$.
\end{itemize}
In fact, given $V\in A_j$, and $r=r\big(\bF_{S^3}(V, A_0)\big)>0$, by Lemma \ref{continuity of push forward by radial flow}, there is $T_V>0$, such that $V(t)\in U_{r}(V)$ for all $0\leq t\leq T_V$.
\begin{claim}
We can choose $T_V$ such that $T_V\geq T_j>0$, where $T_j$ depends only on $j$.
\end{claim}
\begin{proof}
This follows from the compactness of $A_j$ as follows. Assume that there is a sequence $\{V_\al: \al\in\N\}\subset A_j$, such that the maximal possible time $T_{\al}$ for $\{V_{\al}(t)=\big(\Phi_{V_{\al}}(t)\big)_{\#}V_{\al}, 0\leq t\leq T_{\al}\}$ to stay in $U_{r(\ga_{\al})}(V_{\al})$\footnote{$\ga_{\al}=\bF_{S^3}(V_{\al}, A_0)$.} converge to $0$ as $\al\rightarrow\infty$. Up to a subsequence $\{V_{\be}\}$, $\lim_{\be\rightarrow \infty}V_{\be}=V_{\infty}$ for some $V_{\infty}\in A_j$. Denote $V_{\infty}(t)=\big(\Phi_{V_{\infty}}(t)\big)_{\#}V_{\infty}$. For any fixed small $t>0$, by Lemma \ref{properties of HV}, $\Phi_{V_{\be}}(t)$ will converge to $\Phi_{V_{\infty}}(t)$ locally uniformly on $\R^3$; and near $\infty$,
$$\Phi_{V_{\be}}(t)(x)=\sqrt{1+\frac{2 c_{\be}t}{r^2}}x, \quad \Phi_{V_{\infty}}(t)(x)=\sqrt{1+\frac{2 c_{\infty}t}{r^2}}x,$$
for some $c_{\be}, c_{\infty}\in [0, 1]$, with $\lim_{\be\rightarrow \infty}c_{\be}=c_{\infty}$. Using argument similar to that in Lemma \ref{continuity of push forward by radial flow}, given $\lim_{\be\rightarrow\infty}t_{\be}=t_{\infty}$, we can show that $\lim_{\be\rightarrow\infty} V_{\be}(t_{\be})=V_{\infty}(t_{\infty})$. As $V_{\infty}\in A_j$, there exists $T_{\infty}>0$, such that $V_{\infty}$ stays in $U_{r(\ga_{\infty})/2}(V_{\infty})$\footnote{$\ga_{\infty}=\bF_{S^3}(V_{\infty}, A_0)$.} for all $0\leq t\leq T_{\infty}$. Therefore, for $\be$ large enough, $V_{\be}(t)$ will stay in $U_{r(\ga_{\be})}(V_{\be})$ for all $0\leq t\leq T_{\infty}$. It is then a contradiction to the fact that $T_{\be}\rightarrow 0$.
\end{proof}
Therefore, we can choose $T$ satisfying the requirement which is a continuous function depending only on $\bF_{S^3}(V, A_0)$.\\

In summary, given $V\in A\backslash A_0$, with $\ga=\bF_{S^3}(V, A_0)>0$, we can transform $V$ to $V^{\pr}=\big(\Phi_{V}[T(\ga)]\big)_{\#}V$, such that
\begin{itemize}
\vspace{-5pt}
\setlength{\itemindent}{1em}
\addtolength{\itemsep}{-0.7em}
\item[$(\rom{1})$] The map $V\rightarrow V^{\pr}$ is continuous under the $\bF_{S^3}$-metric;

\item[$(\rom{2})$] Using (\ref{1st variation upper bound-continuous version}),
\begin{equation}
\|V^{\pr}\|(S^3)-\|V\|(S^3) \leq \int_{0}^{T(\ga)}[\de_G V(t)] (H_V) dt\leq -T(\ga)g(\ga)<0.
\end{equation}
\end{itemize}

Denote
$$\Psi_V(t, \cdot)=\Phi\big([T(\ga)]t, \cdot\big),\quad \textrm{for } t\in[0, 1];$$
and $L: \R^+\rightarrow \R^+$, with $L(\ga)=T(\ga)g(\ga)$, then $L(0)=0$ and $L(\ga)>0$ if $\ga>0$; and
\begin{equation}\label{decrease mass by isotopy}
\|V^{\pr}\|(S^3)\leq \|V\|(S^3)-L(\ga).
\end{equation}


\subsection{Deforming continuous families by the tightening map}\label{Deforming smooth families by the tightening map}

Given a continuous family of surfaces $\{\Si_{\nu}\}_{\nu\in I^n}$ as in Definition \ref{continuous family}, viewing each slice $\Si_{\nu}$ as an integer multiplicity 2-varifold in $\V_2(\R^3)$, we have
\begin{lemma}
$\{\Si_{\nu}^G\}$ is a continuous family in $\V_2(S^3)$, where $\Si_{\nu}^G=\frac{1}{4\pi}e^{-\frac{|x|^2}{4}}\Si_{\nu}$.
\end{lemma}
\begin{proof}
This is an easy corollary of the locally smoothly continuity of $\Si_{\nu}$ and the continuity of $\|\Si^G_{\nu}\|(S^3)=F(\Si_{\nu})$.
\end{proof}
Using \S \ref{A map from A to the space of isotopies}, we can associate with each slice $\Si^G_{\nu}$ a vector field $H_{\nu}=H_{\Si^{G}_{\nu}}$ and a time $T_{\nu}=T_{\Si^{G}_{\nu}}=T\big(\bF_{S^3}(\Si^G_{\nu}, A_0)\big)$, and an isotopy $\Psi_{\nu}=\Psi_{\Si^G_{\nu}}$. By Lemma \ref{properties of HV}(0), 
$$H_{\nu}=0, \quad \Psi_{\nu}=id,\quad \textrm{when }\nu\in\partial I^n.$$
Define a new family of surfaces $\{\Ga_{\nu}\}$ by
$$\Ga_{\nu}=\Psi_{\Si^G_{\nu}}(1, \Si_{\nu}),\quad\textrm{for all } \nu\in I^n.$$
It is easily seen that
$$\Ga^G_{\nu}=\big(\Psi_{\Si^G_{\nu}}\big)_{\#}\Si^G_{\nu} \quad\textrm{viewed as varifolds as (\ref{push-forward in Gaussian metric})}.$$
Thus by (\ref{decrease mass by isotopy}),
\begin{displaymath}
\begin{split}
F(\Ga_{\nu})=\|\Ga^G_{\nu}\|(S^3) & \leq \|\Si^G_{\nu}\|(S^3)-L\big(\bF_{S^3}(\Si^G_{\nu}, A_0)\big)\\
                                                       & = F(\Si_{\nu})-L\big(\bF_{S^3}(\Si^G_{\nu}, A_0)\big).
\end{split}
\end{displaymath}

However, $\{\Ga_{\nu}\}$ does not necessarily belong to the saturated set $\La$ of $\{\Si_{\nu}\}$, as the family of diffeomorphisms $\psi_{\nu}(x)=\Psi_{\Si^G_{\nu}}(1, x)$ may not depend smoothly on $x$ and $\nu$. Note that $\Psi_{\Si^G_{\nu}}$ is generated by the n-parameter family of vector fields $h_{\nu}=T_{\nu}H_{\nu}=T_{\Si^{G}_{\nu}}H_{\Si^{G}_{\nu}}$, and $h$ is a continuous map
$$h: I^n\rightarrow C^1(\R^3, \R^3) \textrm{ with the $C^1$ topology};$$
$$\textrm{and } h=0 \textrm{ when restricted to } \partial I^n.$$
Using a mollification on the domain $I^n\times\R^3$, we can approximate $h$ by $\ti{h}\in C^{\infty}(I^n\times\R^3, \R^3)$, such that $\sup_{\nu}\|h_{\nu}-\ti{h}_{\nu}\|_{C^1}$ is as small as we want. Moreover, we can make sure that
$$\{\nu\in I^n: \ti{h}_{\nu}\neq 0\} \textrm{ is a compact subset of }\mathring{I}^n.$$
Consider the smooth n-parameter family of isotopies $\ti{\Psi}_{\nu}$ generated by the vector fields $\ti{h}_{\nu}$\footnote{The existence of such $\ti{\Psi}_{\nu}$ follows from standard ODE theory as $\|\ti{h}_{\nu}\|_{C^1}$ is uniformly bounded.}, and the n-parameter family of surfaces $\{\ti{\Ga}_{\nu}\}$ given by $\ti{\Ga}_{\nu}=\ti{\Psi}_{\nu}(1, \Si_{\nu})$.
\begin{lemma}\label{vector fields with C1 bound preserve continuous family}
$\{\ti{\Ga}_{\nu}\}$ is a continuous family as in Definition \ref{continuous family}, and hence lies in the saturated set of $\{\Si_{\nu}\}$.
\end{lemma}
\begin{proof}
The only thing we need to prove is the continuity of $\nu\rightarrow F(\ti{\Ga}_{\nu})$. In fact,
\begin{displaymath}
\begin{split}
F(\ti{\Ga}_{\nu}) & =\int_{\Si_{\nu}}\frac{1}{4\pi}e^{-\frac{|\ti{\Psi}_{\nu}(x)|^2}{4}}J\ti{\Psi}_{\nu}(x, T_x\Si_{\nu})d\mH^2(x)\\
                          & =\int_{\Si_{\nu}}\underbrace{e^{-\big(\frac{|\ti{\Psi}_{\nu}(x)|^2-|x|^2}{4}\big)}J\ti{\Psi}_{\nu}(x, T_x\Si_{\nu})}_{I_{\nu}(x)} \frac{1}{4\pi}e^{-\frac{|x|^2}{4}}d\mH^2(x).
\end{split}
\end{displaymath}
As $\ti{h}_{\nu}$ has uniform $C^1$ bound, $I_{\nu}$ is bounded, i.e. $\|I_{\nu}\|_{\infty}\leq C$. 
Fix ${\nu}_0\in I^n$, and for any $\ep>0$, we can choose a large ball $B_{R_{\ep}}(0)$, such that $\int_{\Si_{\nu_0}\cap B^c_{R_{\ep}}(0)}\frac{1}{4\pi}e^{-\frac{|x|^2}{4}}d\mH^2(x)<\frac{\ep}{4C}$. By the continuity of $F(\Si_{\nu})$ and the continuity of $\Si_{\nu}$ under the locally smoothly topology, for $|\nu-\nu_0|$ small enough, $\int_{\Si_\nu\cap B^c_{R_{\ep}}(0)}\frac{1}{4\pi}e^{-\frac{|x|^2}{4}}d\mH^2(x)<\frac{\ep}{2C}$. Again by the continuity of $\ti{h}_\nu$ and $\Si_\nu$ under the locally smoothly topology, $\big{|}\int_{\Si_\nu\cap B_{R_{\ep}}(0)}I_\nu(x)\frac{1}{4\pi}e^{-\frac{|x|^2}{4}}d\mH^2(x)-\int_{\Si_{\nu_0}\cap B_{R_{\ep}}(0)}I_{\nu_0}(x)\frac{1}{4\pi}e^{-\frac{|x|^2}{4}}d\mH^2(x)\big{|}<\frac{\ep}{4}$ for $|\nu-\nu_0|$ small enough. Hence $|F(\ti{\Ga}_\nu)-F(\ti{\Ga}_{\nu_0})|<\ep$ for $|\nu-\nu_0|$ small enough.
\end{proof}
By similar argument as above, for any given $\ep>0$, we can choose $\sup_\nu\|h_\nu-\ti{h}_\nu\|_{C^1}$ small enough, such that $\sup_\nu |F(\Ga_\nu)-F(\ti{\Ga}_\nu)|\leq \ep$, so that
$$F(\ti{\Ga}_\nu)\leq F(\Si_\nu)-L\big(\bF_{S^3}(\Si^G_\nu, A_0)\big)+\ep.$$
Also, we can make sure that
$$\bF_{S^3}(\Ga_\nu^G, \ti{\Ga}_\nu^G)\leq \ep,$$
by possibly shrinking $\sup_\nu\|h_\nu-\ti{h}_\nu\|_{C^1}$.\\

Now we are ready to carry out the tightening process.

\begin{proof}
(of Proposition \ref{pull tight}) Choose a minimizing sequence $\{\{\Si_\nu\}^k\}\subset\La$, with $\F(\{\Si_\nu\}^k)\leq W(\La)+\frac{1}{k}$. For each $\{\Si_\nu^k\}$, there is a family of isotopies $\{\Psi^k_\nu\}$ and a family of surfaces $\{\Ga^k_\nu\}$ given by $\Ga^k_\nu=\Psi^k_\nu(1, \Si^k_\nu)$. Moreover, by the discussion above we can find a smooth n-parameter family of isotopies $\{\ti{\Psi}_\nu^k\}$, such that $\{\ti{\Ga}_\nu^k\}$ given by $\ti{\Ga}_\nu^k=\ti{\Psi}_\nu^k(1, \Si^k_\nu)$ lies in $\La$, and
\begin{equation}\label{eq 1 in the proof of pull tight}
F(\ti{\Ga}^k_\nu)\leq F(\Si^k_\nu)-L\big(\bF_{S^3}[(\Si^k_\nu)^G, A_0]\big)+\frac{1}{k};
\end{equation}
\begin{equation}\label{eq 2 in the proof of pull tight}
\bF_{S^3}\big[(\Ga^k_\nu)^G, (\ti{\Ga}^k_\nu)^G\big]\leq \frac{1}{k}.
\end{equation}
Therefore $\{\{\ti{\Ga}_\nu\}^k\}$ is also a minimizing sequence in $\La$. If $\{\{\ti{\Ga}_\nu\}^k\}$ does not satisfy the requirement of Proposition \ref{pull tight}, then there exists a min-max sequence $\{\ti{\Ga}^k_{\nu_k}\}$, with $F(\ti{\Ga}^k_{\nu_k})\rightarrow W(\La)$, but $\bF_{S^3}\big((\ti{\Ga}^k_{\nu_k})^G, \bar{A}_0\big)\geq c>0$. Since $W(\La)>\max_{\nu\in\partial I^n}F(\Si_{\nu})$, we have that $F(\ti{\Ga}^k_{\nu_k})>\max_{\nu\in\partial I^n}F(\Si_{\nu})=\max_{V\in B}\|V\|(S^3)$ for $k$ large enough. Therefore we know that $\bF_{S^3}\big((\ti{\Ga}^k_{\nu_k})^G, A_0\big)\geq c^{\pr}>0$.

Let $\nu=\nu_k$ in (\ref{eq 1 in the proof of pull tight}) and take a limit as $k\rightarrow\infty$; we get
$$W(\La)\leq W(\La)-\lim_{k\rightarrow\infty} L\big(\bF_{S^3}[(\Si^k_{\nu_k})^G, A_0]\big).$$
This means that $\lim_{k\rightarrow\infty}\bF_{S^3}\big((\Si^k_{\nu_k})^G, A_0\big)=0$. In \S \ref{A map from A to the space of isotopies}, we know that 
$$\bF_{S^3}\big((\Si^k_{\nu_k})^G, (\Ga^k_{\nu_k})^G\big)\leq r\big(\bF_{S^3}[(\Si^k_{\nu_k})^G, A_0]\big)\rightarrow 0.$$
So $\lim_{k\rightarrow\infty}\bF_{S^3}[(\Ga^k_{\nu_k})^G, A_0]=0$, and this is a contradiction to (\ref{eq 2 in the proof of pull tight}) and the fact $\bF_{S^3}\big((\ti{\Ga}^k_{\nu_k})^G, A_0\big)\geq c^{\pr}>0$. So $\{\{\ti{\Ga}_{\nu}\}^k_{\nu\in I^n}\}$ is a minimizing sequence of $\La$ for which every min-max sequence converge to a $F$-stationary varifold under $\bF_{S^3}$ with no mass supported at $\{\infty\}$. 
\end{proof}




\section{Appendix}


\subsection{Calculation of Jacobian of $f_t$ (\ref{flow of radial vf})}\label{calculation of Jacobian of ft}
Given a 2-plane $S$ with an orthonormal basis $\{e_1, e_2\}$, then
\begin{equation}\label{differential of ft}
\begin{split}
df_t(e_i) & =\sqrt{1+\frac{2t}{r^2}}e_i+\frac{1}{2}\frac{2t\cdot(-2)\cdot\frac{\nabla_{e_i}r}{r^3}}{\sqrt{1+\frac{2t}{r^2}}}x\\
              & =\sqrt{1+\frac{2t}{r^2}}e_i-\frac{2t/r^3}{\sqrt{1+\frac{2t}{r^2}}}\lan e_i, \nabla r\ran x.
\end{split}
\end{equation}
Therefore, the $l$-th component of $d f_t(e_i)$ is given by
$$\big[df_t(e_i)\big]^l=\sqrt{1+\frac{2t}{r^2}}\Big(\lan e_i, \partial_l\ran-\frac{2t/r^3}{1+\frac{2t}{r^2}}\lan e_i, \nabla r\ran x^l\Big),\quad l=1, 2, 3.$$
The matrix $(df_t)^*(df_t)$ is given by:
\begin{displaymath}
\begin{split}
\big[(df_t)^*(df_t)\big]_{ij} & =\sum_{l=1}^3\big[df_t(e_i)\big]^l\big[df_t(e_j)\big]^l\\
                                        & =(1+\frac{2t}{r^2})\Big(\de_{ij}-\frac{4t/r^2}{1+\frac{2t}{r^2}}\lan e_i, \nabla r\ran\lan e_j, \nabla r\ran+\frac{(2t/r^2)^2}{(1+\frac{2t}{r^2})^2}\lan e_i, \nabla r\ran\lan e_j, \nabla r\ran\Big)\\
                                        & =(1+\frac{2t}{r^2})\Big(\de_{ij}-\frac{(1+\frac{2t}{r^2})^2-1}{(1+\frac{2t}{r^2})^2}\lan e_i, \nabla r\ran\lan e_j, \nabla r\ran\Big).
\end{split}
\end{displaymath}
So the determinant of $(df_t)^*(df_t)$ is\footnote{Using the fact that $det(I+uv^T)=1+\lan u, v\ran$, where $u, v\in \R^n$.},
\begin{displaymath}
\begin{split}
det\big[(df_t)^*(df_t)\big] & =(1+\frac{2t}{r^2})^2\Big(1-\frac{(1+\frac{2t}{r^2})^2-1}{(1+\frac{2t}{r^2})^2}|\nabla_S r|^2\Big)\\
                                       & =1+\big[(1+\frac{2t}{r^2})^2-1\big]|\nabla^{\perp}r|^2.
\end{split}
\end{displaymath}
Then, the Jacobian $Jf_t(x, S)$ is given by
\begin{equation}
Jf_t(x, S)= \sqrt{det\big[(df_t)^*(df_t)\big]}=\sqrt{1+\big[(1+\frac{2t}{r^2})^2-1\big]|\nabla^{\perp}r|^2}.
\end{equation}


\subsection{$F$-stationary varifold with bounded entropy has polynomial volume growth}

In this part, we generalize the polynomial volume growth bound for self-shrinkers with bounded Gaussian volume \cite{CZ} to the varifold setting.

Let $V\in\V_n(\R^{n+k})$ be a $F$-stationary n-varifold, i.e. for any $X\in\X_c(\R^{n+k})$,
\begin{equation}\label{F-stationary 2}
\int_{G_n(\R^{n+k})}\big(div_S X-\lan X, \frac{x}{2}\ran\big)e^{-\frac{|x|^2}{4}}dV(x, S)=0.
\end{equation}

\begin{proposition}\label{polynomial volume growth}
Assume that $V$ has bounded Gaussian volume, i.e.
$$F(V)=\frac{1}{(4\pi)^{n/2}}\int_{\R^{n+k}}e^{-\frac{|x|^2}{4}} d\|V\|<\infty.$$
Then $\|V\|$ has n-dimensional Euclidean volume growth, i.e.
$$\|V\|\big(B(0, r)\big)\leq C\cdot F(V)\cdot r^n,$$
where $B(0, r)$ is the Euclidean ball centered at $0$ of radius $r>0$, and $C=e^{1/4}(4\pi)^{n/2}$.
\end{proposition}
\begin{proof}
Denote $f(x)=\frac{|x|^2}{4}$; then $\nabla f=\frac{x}{2}$, and $|\nabla f|^2=f$. For any $n$-plane $S$ in $\R^{n+k}$, $div_S\nabla f=\frac{n}{2}$.

Given a fixed number $r>0$, denote
\begin{equation}\label{I(t)}
I(t)=\frac{1}{t^{n/2}}\int_{\R^{n+k}}\phi(\frac{4f}{r^2})e^{-\frac{f}{t}}d\|V\|(x),
\end{equation}
where $\phi(\cdot)$ is a cutoff function such that: for some $\ep>0$,
\begin{displaymath}
\left. \phi(s)\ \Bigg\{ \begin{array}{ll}
=0, \quad \textrm{ if $s\geq 1$}\\
=1, \quad \textrm{ if $s\leq 1-\ep$}\\
\geq 0, \quad \textrm{ otherwise},
\end{array}\right.
\end{displaymath}
and $\phi^{\pr}(s)\leq 0$.  Then
\begin{equation}\label{derivative of I}
I^{\pr}(t)=\frac{1}{t^{\frac{n}{2}+1}}\int\phi(\frac{4f}{r^2})\big(-\frac{n}{2}+\frac{f}{t}\big)e^{-\frac{f}{t}}d\|V\|.
\end{equation}

Choose a vector field $X\in\X_c(\R^{n+k})$ as
$$X=\phi(\frac{4f}{r^2})e^{-(\frac{1}{t}-1)f}\nabla f.$$
Then in (\ref{F-stationary 2}),
\begin{displaymath}
\begin{split}
div_S X & =\phi^{\pr}(\frac{4f}{r^2})\cdot\frac{4|\nabla_S f|^2}{r^2}e^{-(\frac{1}{t}-1)f}\\
              & +\phi(\frac{4f}{r^2})\big(\underbrace{div_S(\nabla f)}_{=\frac{n}{2}}-(\frac{1}{t}-1)|\nabla_S f|^2\big)e^{-(\frac{1}{t}-1)f},
\end{split}
\end{displaymath}
and as $\nabla f=\frac{x}{2}$
$$\lan X, \frac{x}{2}\ran=\phi(\frac{4f}{r^2})e^{-(\frac{1}{t}-1)f}|\nabla f|^2.$$
Plugging into (\ref{F-stationary 2}), we get
$$\int_{G_n(\R^{n+k})}\Big[\phi^{\pr}(\frac{4f}{r^2})\cdot\frac{4|\nabla_S f|^2}{r^2}+\phi(\frac{4f}{r^2})\big(\frac{n}{2}-(\frac{1}{t}-1)|\nabla_S f|^2-|\nabla f|^2\big)\Big]e^{-\frac{f}{t}}dV(x, S)=0.$$
Plug in $|\nabla f|^2=f$ into the above equation and use (\ref{derivative of I}); we get
$$I^{\pr}(t)=\frac{1}{t^{\frac{n}{2}+1}}\int_{G_n(\R^{n+k})}\Big[\phi^{\pr}(\frac{4f}{r^2})\cdot\frac{4|\nabla_S f|^2}{r^2}+\phi(\frac{4f}{r^2})\cdot(\frac{1}{t}-1)(|\nabla f|^2-|\nabla_S f|^2)\Big]e^{-\frac{f}{t}}d V(x, S).$$
In the above equation, $\phi^{\pr}\leq 0$; $|\nabla f|^2-|\nabla_S f|^2\geq 0$; and when $t\geq 1$, $\frac{1}{t}-1\leq 0$; so
$$I^{\pr}(t)\leq 0\quad\textrm{when $t\geq 1$}.$$
Therefore $I(t)\leq I(1)$; or equivalently
$$\frac{1}{t^{n/2}}\int\phi(\frac{4f}{r^2})e^{-f/t}d\|V\|\leq \int\phi(\frac{4f}{r^2})e^{-f}d\|V\|.$$

Let $\ep\rightarrow 0$ in the definition of $\phi$; or equivalently let $\phi$ approach the characterization function $\chi_{[0, 1]}$; then the above inequality tends to
$$\frac{1}{t^{n/2}}\int_{\{f\leq \frac{r^2}{4}\}}e^{-f/t}d\|V\|(x)\leq \int_{\{f\leq \frac{r^2}{4}\}}e^{-f}d\|V\|\leq (4\pi)^{n/2}F(V).$$
Note that $\{f\leq \frac{r^2}{4}\}=\{|x|\leq r\}$. Let $t=r^2$; then $e^{-f/t}=e^{-\frac{|x|^2}{4r^2}}\geq e^{-\frac{1}{4}}$ on $\{|x|\leq r\}$. Hence the above inequality implies that
$$\frac{1}{r^n}e^{-1/4}\int_{\{|x|\leq r\}}d\|V\|(x)\leq (4\pi)^{n/2}F(V).$$
Hence
$$\|V\|\big(B(0, r)\big)\leq e^{1/4}(4\pi)^{n/2}F(V)r^n.$$
\end{proof}


\subsection{Entropy of $F$-stationary varifold is achieved at $(\vec{0}, 1)$}\label{entropy achieved at (0, 1)}

In this part, we include a standard property for $F$-stationary varifolds with bounded Gaussian volume, which says that the Entropy functional (see (\ref{definition of entropy}) and \cite[(0.6)]{CM12}) is achieved at $(\vec{0}, 1)$. This generalizes \cite[\S 7.2]{CM12} into the varifold setting.

Given $x_0\in\R^{n+1}$, $t_0>0$, the {\em n-dimensional Gaussian density} centered at $(x_0, t_0)$ is a function $\rho_{(x_0, t_0)}: \R^{n+1}\rightarrow\R^+$, defined by
\begin{equation}\label{Gaussian density at (x t)}
\rho_{(x_0, t_0)}(x)=\frac{1}{(4\pi t_0)^{n/2}}e^{-\frac{|x-x_0|^2}{4t_0}}.
\end{equation}
Given $V\in\V_n(\R^{n+1})$, the {\em n-dimensional Gaussian volume centered at $(x_0, t_0)$} is defined by
\begin{equation}\label{Gaussian volume at (x t)}
F_{x_0, t_0}(V)=\int_{\R^{n+1}}\rho_{(x_0, t_0)}(x)d\|V\|.
\end{equation}
The {\em entropy} of $V$ is defined by
\begin{equation}\label{entropy}
\la(V)=\sup_{x_0\in\R^{n+1}, t_0>0}F_{x_0, t_0}(V).
\end{equation}

\begin{theorem}\label{entropy achieve at (0 1)}
Given $V\in\V_n(\R^{n+1})$ with bounded Gaussian volume, i.e. 
$$F(V)=\int_{\R^{n+1}}\rho_{(0, 1)}(x)d\|V\|<\infty,$$
assume that $V$ is $F$-stationary, i.e. for any $X\in\X_c(\R^{n+1})$,
\begin{equation}\label{F-stationary 3}
\int_{G_n(\R^{n+1})}\big(div_S X-\lan X, \frac{x}{2}\ran\big)\rho_{(0, 1)}(x)dV(x, S)=0.
\end{equation}
Then the entropy functional achieves a global maximum at $(0, 1)$, i.e.
$$\la(V)=F_{0, 1}(V)\big(=F(V)\big).$$
\end{theorem}
\begin{remark}
The result is used in the proof of Lemma \ref{A0 is compact}. In fact, we only need the fact that $F_{0, t}(V)$ achieves a global maximum at $t=1$. For completeness we prove the above stronger version. The proof is adapted from \cite[\S 7.2]{CM12}.
\end{remark}
\begin{proof}
This follows from the following lemma.
\end{proof}

\begin{lemma}\label{part 2 of entropy achieve at (0 1)}
Let $V$ be $F$-stationary with $F(V)<\infty$. Given $y\in\R^{n+1}$, $a\in\R$, let
$$g(s)=F_{sy, 1+as^2}(V);$$
then $g^{\pr}(s)\leq 0$ for all $s>0$ with $1+as^2>0$.
\end{lemma}
\begin{proof}
As $\|V\|$ has polynomial volume growth by Proposition \ref{polynomial volume growth}, $g(s)$ is a smooth function on $s$. Denote $\rho_s(x)=\rho_{(sy, 1+as^2)}(x)$. By straightforward calculation,
\begin{displaymath}
\begin{split}
g^{\pr}(s) & =\int\Big[-\frac{nas}{1+as^2}+\frac{\lan y, x-sy\ran}{2(1+as^2)}+\frac{|x-sy|^2as}{2(1+as^2)^2}\Big]\rho_s(x) d\|V\|(x)\\
                & =\frac{1}{1+as^2}\int\Big[-nas+\frac{\lan x-sy, asx+y\ran}{2(1+as^2)}\Big]\rho_s(x) d\|V\|(x).
\end{split}
\end{displaymath}

Consider the vector field
$$X=-(asx+y)e^{-\big(\frac{|x-sy|^2}{4(1+as^2)}-\frac{|x|^2}{4}\big)}.$$
Using the polynomial volume growth of $\|V\|$ and the explicit asymptotic behavior of $X$, we can still plug $X$ into (\ref{F-stationary 3}) by using compactly supported approximations of $X$. In particular, given an $n$-plane $S$ as above,
\begin{displaymath}
\begin{split}
div_S X & =\Big(-nas+\big{\lan} asx+y, P_S\big(\frac{x-sy}{2(1+as^2)}-\frac{x}{2}\big)\big{\ran}\Big)e^{-\big(\frac{|x-sy|^2}{4(1+as^2)}-\frac{|x|^2}{4}\big)}\\
              & =\Big(-nas-\big{\lan} asx+y, P_S \frac{(asx+y)s}{2(1+as^2)}\big{\ran}\Big)e^{-\big(\frac{|x-sy|^2}{4(1+as^2)}-\frac{|x|^2}{4}\big)}\\
              & =\Big(-nas-\frac{|P_S(asx+y)|^2s}{2(1+as^2)}\Big)e^{-\big(\frac{|x-sy|^2}{4(1+as^2)}-\frac{|x|^2}{4}\big)}.
\end{split}
\end{displaymath}
So by plugging $X$ to (\ref{F-stationary 3}), we get
$$\int\Big(-nas-\frac{|P_S(asx+y)|^2s}{2(1+as^2)}+\frac{1}{2}\lan asx+y, x\ran\Big)e^{-\frac{|x-sy|^2}{4(1+as^2)}}dV(x, S)=0.$$
Multiplying the above equation with $\frac{1}{(4\pi(1+as^2))^{n/2}}\cdot\frac{1}{1+as^2}$ and subtracting with $g^{\pr}(s)$, we get
\begin{displaymath}
\begin{split}
g^{\prime}(s) & =\frac{1}{2(1+as^2)}\int\big[\frac{\lan x-sy, asx+y\ran}{1+as^2}-\lan asx+y, x\ran+\frac{|P_S(asx+y)|^2s}{1+as^2}\big]\rho_s dV(x, S)\\
                      & =\frac{1}{2(1+as^2)}\int\big[-\frac{|asx+y|^2s}{1+as^2}+\frac{|P_S(asx+y)|^2s}{1+as^2}\big]\rho_s dV(x, S)\\
                      & =\frac{1}{2(1+as^2)}\int\big[-\frac{|(asx+y)^{\perp}|^2s}{1+as^2}\big]\rho_s dV(x, S)\\
                      & \leq 0.                 
\end{split}
\end{displaymath}
\end{proof}



\parindent 0ex

Daniel Ketover, Fine Hall, Princeton University; dketover@math.princeton.edu.\\

Xin Zhou, Department of Mathematics, Massachusetts Institute of Technology, 77 Massachusetts Avenue, Cambridge, MA 02139-4307; xinzhou@math.mit.edu.

\end{document}